\documentclass[11pt]{article}
\usepackage{amsmath}
\usepackage{amsfonts}
\usepackage{amsthm}
\usepackage{amssymb, setspace}
\usepackage{color,graphicx,geometry,hyperref,fancyhdr}
\usepackage[T1]{fontenc}
\usepackage{float}
\usepackage{subfig}
\usepackage{commath}
\usepackage{bbm}
\usepackage{mathrsfs,fleqn}
\usepackage{tikz}
\usetikzlibrary{decorations.pathreplacing}
\usepackage{wrapfig}
\usepackage{verbatim}

\newtheorem{theorem}{Theorem}[section]

\newcommand{\N}{\mathbb{N}}

\newcommand{\C}{\mathbb{C}}

\graphicspath{{paper-pics/}}

\begin{document}

\begin{flushleft}
\Large 
\noindent{\bf \Large Qualitative reconstruction methods for imaging interior Robin interfaces in EIT from Robin-to-Dirichlet data}
\end{flushleft}

\vspace{0.2in}

{\bf  \large Rafael Ceja Ayala}\\
\indent {\small School of Mathematical and Statistical Sciences, Arizona State University, Tempe, AZ 85287 }\\
    \indent {\small Email: \texttt{rcejaaya@asu.edu}}\\

{\bf  \large Malena I. Español}\\
\indent {\small School of Mathematical and Statistical Sciences, Arizona State University, Tempe, AZ 85287 }\\
\indent {\small Email: \texttt{mespanol@asu.edu} }\\

{\bf  \large Govanni Granados}\\
\indent {\small Department of Mathematics,
University of North Carolina, Chapel Hill, NC 27599 }\\
\indent {\small Email:  \texttt{ggranad@unc.edu} }\\

\begin{abstract}
\noindent We consider an inverse shape problem arising in electrical impedance tomography (EIT) for nondestructive testing, in which interior defects are modeled through Robin transmission conditions. Unlike classical formulations, we impose Robin boundary conditions on both the exterior measurement surface and the interior interface, and use the Robin-to-Dirichlet (RtD) map as the available data. Within this setting, we develop qualitative (non-iterative) reconstruction methods based on the Linear Sampling Method (LSM) and the Regularized Factorization Method (RFM), and derive new analytical characterizations that enable these methods to identify interior regions. We further propose a numerical implementation that incorporates regularization strategies and demonstrate, through experiments, that the methods reliably reconstruct interior regions of interest. 
\end{abstract}

\noindent {\bf Keywords}: electrical impedance tomography, regularized factorization method, linear sampling method, Robin transmission condition. \\ 

\noindent {\bf MSC}: 35R30, 35J25, 47A52, 65N21

\section{Introduction}
The problem we study in this paper is motivated by applications of electrical impedance tomography (EIT)~\cite{eit-review, EIT-cheney, mueller-book} for nondestructive testing, in which the goal is to identify internal structures within a material using electrostatic measurements taken at its surface. In practice, such measurements are often affected by the electrode-surface contact impedance, which is well modeled by Robin boundary conditions (see \cite{harris2,harris1} for similar problems in EIT). This leads to an inverse shape problem in which the available data are given by the Robin-to-Dirichlet (RtD) map, which assigns to each applied impedance (Robin) boundary data the corresponding voltage distribution on the boundary.

In particular, we focus on recovering an interior region where a Robin transmission condition is imposed, which models corrosion on the boundary of an inclusion within the material. In the context of nondestructive testing, such internal defects must be identified without compromising the material’s overall structure, especially in industrial settings where removing or damaging the surrounding healthy material is not feasible, safe, or practical~\cite{tallman2020structural}. These considerations motivate the use of qualitative reconstruction methods, which are designed to infer internal features from boundary data without relying on heavy prior information. Such methods have been successfully employed in several tomography inverse problems \cite{GLSM, DSM-EIT, DSM-DOT,EIT-FM,MUSIC-Hanke, MUSIC-Hanke2, harris1, MUSIC-sweep}.

While much of the previous literature in EIT has addressed inverse parameter problems using the Neumann-to-Dirichlet (NtD) (e.g., \cite{eit-transmission2,eit-transmission1}) or Dirichlet-to-Neumann (DtN) maps (e.g. \cite{EIT-granados1, EIT-granados2}), our focus is on the inverse shape problem using the RtD map. The primary contribution of this work is to analytically show that the RtD map uniquely determines the location and shape of the interior region, and to numerically implement the corresponding sampling algorithms for reconstructing these regions. More specifically, we propose two qualitative reconstruction approaches that require factorizations of the data operator, which depends on the RtD map. These methods are non-iterative and require minimal a priori knowledge about the unknown region. Unlike iterative methods, which often rely on a good initial guess and repeated solutions of forward and adjoint problems, our methods identify the region of interest by linking it to the range or spectral structure of the data operator. This leads to computationally efficient algorithms based on spectral or singular value decompositions of the measured data operator.

This work extends existing shape reconstruction techniques to settings in which Robin-type conditions govern both the measurement boundary and the interior interface. Such settings are more physically realistic and introduce new analytical and computational challenges that we address by developing robust and efficient qualitative reconstruction methods. Specifically, we adapt two classical methods: the Linear Sampling Method (LSM) and a variant of the Factorization Method, namely the Regularized Factorization Method (RFM), to this Robin-Dirichlet framework. These methods have been widely used in inverse problems for both source identification and shape reconstruction, and we demonstrate their effective application to the present setting. For previous works in LSM see \cite{GLSM, cakoni-colton, cakoni-colton-haddar, cheney1, colton-haddar-monk} and for RFM see \cite{cakoni-harris, EIT-granados1, EIT-granados2, harris1}.

Both LSM and RFM recover the region of interest by testing sampling points in the domain through the solution of a discrete ill-posed problem derived from a factorization of the data operator. While LSM relies on a factorization, RFM employs a more refined decomposition of the data operator constructed from the same boundary measurements, yielding a stronger analytical result. For both methods, we adopt filter functions corresponding to the regularization schemes: Spectral cutoff, Tikhonov, and Truncated Total Least Squares (TTLS). Ultimately, we construct an imaging functional for each method exhibiting binary behavior, with large values inside and small values outside the region of interest. Notably, both LSM and RFM rely solely on boundary data to solve the inverse shape problem (see \cite{harris2025sampling} for related applications to recovering buried corroded boundaries using Cauchy data).

The rest of the paper is organized as follows. In Section~\ref{dp-ip}, we rigorously define the direct and inverse problems under consideration. We first establish the well-posedness of the direct problem and define the data operator, which is used to derive our imaging functionals for both the LSM and RFM. Next, we derive an initial factorization of this operator to apply the LSM in Section~\ref{LSM}. Subsequently, we derive a more detailed factorization of the data operator to apply the RFM in Section~\ref{RegFM}. This method follows the theory developed in \cite{harris1}, and yields an efficient imaging functional for reconstructing the shape of the interior regions. In Section~\ref{section: numerics}, numerical examples in the unit disk in $\mathbb{R}^2$ are presented to validate the analysis of the proposed imaging functionals. Lastly, in Section~\ref{end}, we conclude the paper by summarizing the results and outlining possible future directions.

\section{The direct and inverse problem}\label{dp-ip}
We start by studying the direct problem associated with electrostatic imaging in the presence of an interior, corroded interface. This problem incorporates an exterior impedance (Robin) boundary condition and an interior Robin transmission condition. Assume that $ \Omega \subset \mathbb{R}^d$, where $d=2$ or $d=3$, is a simply connected open domain with Lipschitz boundary $\partial \Omega$. Let $D \subset \Omega$ be an open (potentially multiple) connected subset with $\mathcal{C}^2$ boundary. We further assume that $\text{dist}(\partial \Omega, \overline{D}) > 0$. Thus, the electrostatic potential $u \in H^{1} (\Omega)$ satisfies \\
\begin{wrapfigure}{r}{0.40\textwidth}
    \centering
    \vspace{-1em}
    \begin{tikzpicture}[scale=1.5]
        \draw[thick] (0,0) ellipse (1.7 and 1.35);
        \node at (-1.3,0.3) {\(\Omega\)};
        \node at (1.6,0.9) {\(\partial\Omega\)};

        \draw[thick, dashed] (0,0) circle(0.6);
        \node at (0,0) {\(D\)};
        \node at (0.8,0.1) {\(\partial D\)};

        \draw[->, thick] 
            (1.0, -0.25) .. controls (0.85, -0.5) and (0.7, -0.35) .. (0.66, -0.15);
        \node at (0.9, -0.5) {\scriptsize $+$};

        \draw[->, thick] 
            (0.1, -0.45) .. controls (0.2, -0.5) and (0.4, -0.4) .. (0.5, -0.2);
        \node at (0.25, -0.35) {\scriptsize $-$};
    \end{tikzpicture}
    \vspace{-50pt}
\end{wrapfigure}
\vspace{-25pt}  

\begin{equation}\label{mpde}
\begin{cases}
-\nabla \cdot (\sigma \nabla u) = 0, & \text{in } \Omega \setminus  \partial D, \\[1.2ex]
\sigma\, \partial_{\nu} u + u = f, & \text{on } \partial \Omega, \\[1.2ex]
\left[\!\left[\sigma \partial_{\nu} u\right]\!\right] = \gamma u, & \text{on } \partial D,
\end{cases}
\end{equation}
where \[\left[\!\left[\sigma \partial_{\nu} u\right]\!\right] :=  (\sigma \partial_{\nu} u)^{+} - (\sigma \partial_{\nu} u)^{-},\]
for a given $f \in H^{-1/2} ( \partial\Omega )$. For the rest of the paper, we denote $\nu$ as the unit outward normal on the boundaries $\partial D$ and $\partial \Omega$. The `$+$' notation indicates the trace from the exterior domain $\Omega \setminus \overline{D}$, while the `$-$'  corresponds to the trace from the interior region $D$. The Robin transmission condition given in \eqref{mpde} models the jump in current across the interface $\partial D$ and states that it is proportional to the electrostatic potential $u$. Note that under the assumption that $u \in H^1(\Omega)$, it follows that $ [\![u ]\!] = 0$ on $\partial D$ due to the fact that any function in $H^1(\Omega)$ possesses equal exterior and interior trace across any interface within $\Omega$. 

The boundary condition on $\partial \Omega$ is motivated by the complete electrode model (see, e.g., \cite{completedEITmodel}), where a finite number of electrodes are placed on the exterior boundary. Thus, \eqref{mpde} is the continuous analog of the complete electrode model as we assume continuous contact of an electrode on $\partial \Omega$; for exposition, we assume the contact impedance is 1. This motivates the use of the Robin-to-Dirichlet map in the construction of the data operator, to be defined shortly, for recovering the interior region $D$. 

We also assume that the transmission parameter $\gamma \in L^{\infty} ( \partial D )$ and the conductivity parameter $\sigma > 0$ is constant. For the analysis of well-posedness and the forthcoming investigation of the inverse problem, we assume that there exist constants $\gamma_{\text{min}}$ and $\gamma_{\text{max}}$ such that 
$$0< \gamma_{\text{min}} \leq  \gamma (x) \leq \gamma_{\text{max}}  \quad \text{for a.e. } x \in \partial D. $$ 

The direct problem corresponding to \eqref{mpde} can be stated as follows: given impedance $f$, determine the resulting electrostatic potential of the medium $\Omega$ containing the inclusion and interface. We proceed to show that the boundary value problem \eqref{mpde} is well-posed for any given $f \in H^{-1/2} ( \partial \Omega)$. To this end, we consider Green's 1st Theorem on the region $\Omega \backslash \overline{D}$ and $D$ separately, giving us the following integrals: 
$$\int_{\Omega \backslash \overline{D}} \sigma \nabla u \cdot \nabla \overline{\varphi} \, \text{d}x = \int_{\partial \Omega}   \overline{\varphi} \sigma  \partial_{\nu} u \, \text{d}s - \int_{\partial D} \overline{\varphi}  (\sigma \partial_{\nu} u)^{+} \, \text{d}s $$
and 
$$ \int_{D} \sigma \nabla u \cdot \nabla \overline{\varphi} \, \text{d}x = \int_{\partial D} \overline{\varphi}  (\sigma \partial_{\nu} u)^{-} \, \text{d}s $$ 
for any test function $\varphi \in H^1 (\Omega)$. By adding these two integral equations, we obtain
\[ \int_{\Omega} \sigma \nabla u \cdot \nabla \overline{\varphi} \, \text{d}x=  \int_{\partial \Omega}  \overline{\varphi} \sigma \partial_{\nu} u  \, \text{d}s -  \int_{\partial D} \overline{\varphi} \gamma u  \, \text{d}s,\]
where we have used the Robin transmission condition on $\partial D$. By applying the impedance boundary condition on $\partial \Omega$, one obtains
\begin{equation}\label{vf}
    \int_{\Omega} \sigma \nabla u \cdot \nabla \overline{\varphi} \, \text{d}x +  \int_{\partial \Omega}  \overline{\varphi}  u  \, \text{d}s + \int_{\partial D} \overline{\varphi} \gamma u  \, \text{d}s= \int_{\partial \Omega}  \overline{\varphi} f   \, \text{d}s.
\end{equation}
We use \eqref{vf} to define the sesquilinear form $A(\cdot , \cdot):H^{1}(\Omega) \times H^{1}(\Omega) \mapsto \mathbb{C}$ as 
\[A(u,\varphi)= \int_{\Omega} \sigma\nabla u \cdot \nabla \overline{\varphi} \, \text{d}x +  \int_{\partial \Omega}  \overline{\varphi}  u  \, \text{d}s + \int_{\partial D} \overline{\varphi} \gamma u  \, \text{d}s\]
and the conjugate linear functional $\ell (\cdot): H^{1}(\Omega) \mapsto \mathbb{C}$ as
\[\ell(\varphi)=\int_{\partial \Omega}  \overline{\varphi} f   \, \text{d}s.\]
Thus, the variational formulation of \eqref{mpde} can be rewritten as: find $u \in H^{1}(\Omega)$ such that 
\[A(u, \varphi) = \ell (\varphi) \quad \text{for all} \quad \varphi \in H^{1}(\Omega).\]
It is clear that the sesquilinear form $A(\cdot, \cdot)$ is bounded, whereas its coercivity on $H^1(\Omega)$ can be shown by appealing to a standard Poincaré-type argument, which implies that
$$\| \cdot \|_{H^1(\Omega)}^2 \quad \text{is equivalent to } \quad \int_{\Omega} |\nabla\cdot|^2 \, \text{d}x+\int_{\partial \Omega} |\cdot|^2 \, \text{d} s.$$
See \cite[Chapter~8]{salsa} for more details. It also holds that $\ell(\cdot)$ is a conjugate linear and bounded functional on  $H^{1} (\Omega)$. By the Trace Theorem, we have that 
$$| A(u, \varphi ) | \leq C \|{f}\|_{H^{-1/2} (\partial\Omega)} \|{\varphi}\|_{H^{1} (\Omega)}.$$ Thus, by the Lax-Milgram Lemma, there is a unique solution $u \in H^{1}(\Omega)$ of \eqref{mpde} satisfying 
$$\|{u}\|_{H^{1} (\Omega)} \leq C \|{f}\|_{H^{-1/2} (\partial \Omega)},$$ 
for a constant $C>0$. The above analysis gives the following result.

\begin{theorem}\label{thm:soln-op}
The solution operator corresponding to the boundary value problem \eqref{mpde} $f \mapsto u$ is a bounded linear mapping from $H^{-1/2}(\partial \Omega)$ to $H^{1}(\Omega)$.
\end{theorem}

With the well-posedness of \eqref{mpde} established, we now turn to the formulation of the inverse shape problem. Our results will focus on comparing voltage measurements with and without an inclusion. To this end, we consider the unperturbed Robin problem in $\Omega$ (without an inclusion): given $f\in H^{-1/2}(\partial \Omega)$, determine $u_0\in H^1(\Omega)$ such that 
\begin{equation}\label{bvp1}
    -\nabla\cdot(\sigma \nabla u_0) = 0 \quad \text{in} \quad \Omega \quad \text{with} \quad \sigma \partial_{\nu}u_0+u_0 = f\quad \text{on}\quad \partial  \Omega.
\end{equation} 
Similarly as \eqref{mpde}, one can show the well-posedness of \eqref{bvp1} with stability estimate 
$$\|{u_0}\|_{H^{1} (\Omega)} \leq C \|{f}\|_{H^{-1/2} (\partial \Omega)},$$ 
via the Lax-Milgram Lemma. To reiterate, $u_0$, the solution of \eqref{bvp1}, represents the electrostatic potential of $\Omega$ without an interior region. That is, there is no corrosion on an interior interface. Furthermore, $u_0$ is assumed to be known since $\Omega$ is given.

To define the data operator, we assume that Robin boundary data $f$ is applied to $\partial \Omega$ and the induced voltages are given by $u \rvert_{\partial \Omega}$ and $u_0 \rvert_{\partial \Omega}$, respectively. A critical difference between these Dirichlet traces is that $u \rvert_{\partial \Omega}$ is measured data, whereas $u_0 \rvert_{\partial \Omega}$ can be computed. We define the {\it Robin-to-Dirichlet} operators $M \enspace \text{and} \enspace M_0 \colon H^{-1/2} ( \partial \Omega) \rightarrow H^{1/2} ( \partial\Omega)$ given by
\begin{equation*}\label{RtDoperators}
 M f =  u \big|_{\partial \Omega} \quad \text{and} \quad M_{0} f =  u_0 \big|_{\partial \Omega},
\end{equation*}
where $u$ and $u_0$ are the unique solutions to \eqref{mpde} and \eqref{bvp1}, respectively. It has been shown in \cite{RtD} that the operators are well-defined. By Theorem \ref{thm:soln-op}, the well-posedness of \eqref{bvp1}, and the Trace Theorem, the RtD mappings are bounded linear operators. Thus, we define the data operator as the difference between $M$ and $M_0$. More precisely,
\begin{equation}\label{op:M-M0}
    (M-M_0)\colon H^{-1/2} ( \partial \Omega) \rightarrow H^{1/2} ( \partial\Omega) \quad \text{given by} \quad (M-M_0)f = u-u_0 \big \rvert_{\partial \Omega}.
\end{equation} 
The inverse problem we consider in this work can be stated as follows: determine the shape and location of $D$ using impedance-voltage data provided by the difference of the Robin-to-Dirichlet maps, $(M-M_0)$. 

From voltage data, we will develop two qualitative sampling methods to detect the interior region $D$ without requiring prior knowledge of the transmission parameter $\gamma$ or any assumptions about the number of interior regions. In the following section, we provide the analysis and derivation of our reconstruction methods.

\section{\textbf{Reconstruction Methods}} \label{Factorization section}
In this section, we focus on the preliminaries required to develop the two qualitative methods for reconstructing the interior region $D$ from knowledge of $M$: LSM and RFM. We begin by introducing the necessary operators and outlining a series of results to apply the LSM and the RFM to solve the inverse shape problem. Our goal in this section is to factorize the data operator $(M - M_0)$ in two ways: a first factorization sufficient for the LSM (Section 3.1), and a refined symmetric factorization needed for the RFM (Section 3.2). To this end, we will introduce and study operators that connect jumps in the current across $\partial D$ to boundary traces on $\partial \Omega$. We will show that the RFM is mathematically more advantageous than the LSM, as it provides necessary and sufficient conditions for recovering $D$. We now begin the factorization of $(M-M_0)$ to obtain a more explicit relationship between surface measurements and $D$. During this process, we will describe some important properties of the components in the factorization. 

Motivated by the data operator $(M - M_0 )$, we note that $u - u_0 \in H^1(\Omega)$ solves 
\begin{align*}
    - \nabla\cdot \sigma \nabla (u-u_0 ) &= 0 \quad \text{in} \quad \Omega \setminus \partial D \quad  \text{with} \quad \sigma \partial_{\nu}(u-u_0 )+(u-u_0 ) = 0\quad \text{on}\quad \partial  \Omega\\
&\text{and} \quad  [\![\sigma\partial_\nu (u - u_0) ]\!] \big|_{\partial D} = \gamma u\big|_{\partial D} .
\end{align*} 
So, we present an auxiliary problem where $w\in H^{1} (\Omega )$ solves
\begin{align}  \label{w}
- \nabla \cdot(\sigma\nabla &w) = 0 \quad \text{in} \quad \Omega \setminus \partial D \quad \text{with} \quad \sigma \partial_{\nu}w +w = 0\quad \text{on}\quad \partial  \Omega\\
&\text{and}\quad \quad [\![\sigma\partial_\nu w ]\!]\big|_{\partial D} = \gamma h \nonumber ,
\end{align}
for any given $h \in L^{2} ( \partial D)$. The well-posedness of \eqref{w} follows a similar variational argument used in Section \ref{dp-ip}. As a consequence, we define the bounded linear Source-to-Dirichlet (StD) operator 
\begin{equation}\label{op:G}
    G \colon L^{2} ( \partial D ) \rightarrow H^{1/2} (\partial \Omega ) \quad \text{given by} \quad Gh = w \big \rvert_{\partial \Omega},
\end{equation}
where $w$ is the unique solution of \eqref{w}. The boundedness of $G$ follows from the well-posedness of \eqref{w}. Recall that we aim to recover $D$ from physical measurements on $\partial \Omega$. To this end, we connect the StD operator $G$ to the data operator $(M-M_0)$. Observe that by the well-posedness of \eqref{w}, it also holds that
$$w \big \rvert_{\partial \Omega}= ( M - M_0 ) f \quad \text{ provided that } \quad h = u \big|_{\partial D},$$ 
where $u$ is the unique solution to \eqref{mpde}. This motivates defining the solution operator of \eqref{mpde} as
\begin{equation}\label{op:S}
    S \colon H^{-1/2} ( \partial \Omega ) \rightarrow L^{2} (\partial D ) \quad \text{given by} \quad Sf = u \big|_{\partial D}.
\end{equation}
The StD operator $G$ composed with the solution operator $S$ provides a preliminary decomposition of the data operator. That is, we have that 
\begin{equation}\label{eqn:GSf}
(M-M_0)f = GSf
\end{equation}
for any $f \in H^{-1/2} ( \partial \Omega)$. We note that the left-hand-side of \eqref{eqn:GSf} has an implicit dependency on $D$ given that $(M-M_0)$ maps from $H^{-1/2}(\partial \Omega)$ to $H^{1/2}(\partial \Omega)$. Whereas, the right-hand-side of \eqref{eqn:GSf} is a factorization involving $D$, namely, $L^{2}(\partial D)$.
The following results demonstrate some key analytical properties of the StD operator $G$ needed to further describe the data operator $(M - M_0)$.
\begin{theorem}\label{thm:G-injective}
    The StD operator $G$ as defined in \eqref{op:G} is injective. 
\end{theorem}
\begin{proof}
    Suppose that $G h = 0$. Then, $w$ satisfies $-\nabla \cdot (\sigma \nabla w ) = 0$ in $\Omega \setminus D$ with $w = \partial_{\nu} w = 0 $ on $\partial \Omega$. By Holmgren's Theorem~\cite{holmgren}, this implies that $w = 0$ in $\Omega \setminus D$. Thus, $(\sigma \partial_{\nu} w)^+ \rvert_{\partial D} = 0$. It also follows that $-\nabla \cdot (\sigma \nabla w ) = 0$ in $D$ with $w = 0$ on $\partial D$. Thus, $w = 0$ in $D$ and $(\sigma \partial_{\nu}w)^- \rvert_{\partial D} = 0$, which proves the claim.
\end{proof}

To further analyze the StD operator $G$, we will compute its adjoint, $G^{*}$. To do so, we first define the sesquilinear dual-product
\begin{equation}
    \langle \varphi , \psi \rangle_{\partial \Omega} = \int_{\partial \Omega} \varphi \overline{\psi} \, \text{d}s \quad \text{for all} \quad \varphi \in H^{1/2}(\partial \Omega) \enspace \text{and} \enspace \psi \in H^{-1/2}(\partial \Omega)
\end{equation}
between the Hilbert Space $ H^{1/2}(\partial \Omega)$ and its dual space $ H^{-1/2}(\partial \Omega)$, where $L^{2}(\partial \Omega)$ is their Hilbert pivot space. Then, the adjoint operator $G^*$ will be a mapping from $H^{-1/2}(\partial \Omega)$ into $L^{2}(\partial D)$, which can be derived via the equality
$$\langle G \varphi , \psi \rangle_{\partial \Omega} = ( \varphi , G^* \psi )_{L^{2}(\partial D)} \quad \text{for all} \enspace \varphi \in L^{2}(\partial D) \enspace \text{and} \enspace \psi \in H^{-1/2}(\partial \Omega).$$
Our next result describes the associated boundary-value problem and some important properties of the adjoint $G^*$.

\begin{theorem}\label{CompactG}
    The adjoint $G^{*}\colon H^{-1/2}(\partial \Omega) \rightarrow L^{2}(\partial D)$ is given by $G^{*} \psi = - \gamma v$, where $v \in H^{1}(\Omega)$ satisfies
    \begin{equation*}
        -\nabla \cdot \sigma \nabla v = 0 \enspace \text{in} \enspace \Omega \setminus \partial D \quad \text{with} \quad \sigma \partial_{\nu}v + v = \psi \enspace \text{on} \enspace \partial \Omega \quad \text{and} \quad \left[\!\left[ \sigma \partial_{\nu} v\right]\!\right] = 0 \enspace \text{on} \enspace \partial D.
    \end{equation*}
    Furthermore, $G^{*}$ is compact with dense range.
\end{theorem}
\begin{proof}
    By appealing to Green's 2nd Theorem, it holds that 
    $$0 = \int_{\partial \Omega} w \overline{\sigma \partial_{\nu}v} - \overline{v} \sigma \partial_{\nu} w\, \text{d}s - \int_{\partial D} w \big[ \overline{(\sigma \partial_{\nu} v)^+} - \overline{(\sigma \partial_{\nu} v)^-} \big] \, \text{d}s + \int_{\partial D} \overline{v} \big[ (\sigma \partial_{\nu} w)^{+} - (\sigma\partial_{\nu} w)^{-} \big] \, \text{d}s.$$
    By the boundary conditions on $\partial \Omega$ and $\partial D$ for $v$, this reduces to
    $$\int_{\partial \Omega} w \overline{\psi} \, \text{d}s = - \int_{\partial D} \overline{v} \left[\!\left[\sigma \partial_{\nu} w \right]\!\right] \, \text{d}s.$$
    Using the boundary condition on $\partial D$ for $w$, we get that
    $$ \int_{\partial \Omega} w \overline{\psi} \, \text{d}s = - \int_{\partial D} \gamma h\overline{v} \, \text{d}s.$$
    Thus, we have that 
    $$ \langle Gh , \psi \rangle_{\partial \Omega} = \int_{\partial \Omega} w \overline{\psi} \, \text{d}s = - \int_{\partial D} \gamma h\overline{v} \, \text{d}s = (h, G^{*}\psi) $$
    for all $\psi \in H^{-1/2}(\partial \Omega)$ and $h \in L^{2}(\partial D)$, which implies that $G^{*}\psi = - \gamma v$. To show injectivity, suppose that $G^{*}\psi = 0$, which implies that $v=0$ in $\overline{D}$ given that $\gamma > 0$ for a.e. $x \in \partial D$. Thus, by the boundary condition on $\partial D$, we have that $\partial_{\nu}v^{+} = 0$ on $\partial D$. By Holmgren's Theorem, $v = 0$ in $\Omega$. Therefore, $v = \sigma \partial_{\nu} v = 0$ on $\partial \Omega$, proving the claim. Furthermore, notice that the compact embedding of $H^{1/2}(\partial D)$ into $L^{2}(\partial D)$ implies that $G^{*}$ is compact.
\end{proof}

Given these useful properties of the operator $G$, we still require additional results regarding the solution operator $S$ to derive the LSM and RFM for our inverse problem. In the following results, we show that $S$ is injective and analyze its adjoint. 

\begin{theorem}\label{injectiveS} The solution operator $S$ as defined in \eqref{op:S} is injective.
\end{theorem}
\begin{proof}
    Let $Sf = 0$ which implies that $u =0$ in $\partial D.$ So as $-\nabla\cdot \sigma \nabla u=0$ in $D$ and $u=0$ on $\partial D$, then $u=0$ in $\overline{D}$. Thus, $\partial_{\nu}u\big\rvert_{\partial D}^{-} =0$. By the boundary condition $[\![\sigma \partial_\nu u]\!]=0$ on $\partial D$ and the fact that $\sigma\neq 0$, we have that $\partial_{\nu}u\big\rvert_{\partial D}^{+} =0$. Observe that $-\nabla\cdot \sigma \nabla u=0$ in $\Omega\setminus\overline{D}$, $u=0$ on $\partial D$  and $\partial_{\nu}u\big\rvert^+=0$. By Holmgren's Theorem, we have that $u = 0$ in $\Omega$. Then by the Trace Theorem, $u=0$ on $\partial\Omega$. Furthermore, $\partial_{\nu} u=0$ on $\partial \Omega$. Therefore, $f = 0$ on $\partial \Omega$, proving that $S$ is injective. \\
\end{proof}
We will show that the operator $S$ has a dense range by examining its adjoint. To this end, we present the following theorem, which characterizes the adjoint of $S$ and will be instrumental in our upcoming analysis.

\begin{theorem} \label{adjoint}
The adjoint operator $S^{*}\colon L^2 ( \partial D) \rightarrow H^{1/2} ( \partial \Omega )$ is given by $S^{*}g = v\big|_{\partial \Omega}$, where $v \in H^1 (\Omega)$ satisfies 
\begin{align} \label{adv}
- \nabla\cdot\sigma\nabla &v = 0 \quad \text{in} \quad \Omega \setminus \partial D \quad \text{with} \quad \sigma \partial_{\nu}v +v = 0\quad \text{on}\quad \partial  \Omega\\
&\text{and}\quad \quad [\![\sigma\partial_\nu v ]\!]\big|_{\partial D} = \gamma v- g \nonumber.
\end{align} 
\end{theorem}
\begin{proof} 
Observe that a variational approach allows us to demonstrate the existence and uniqueness of the solution $v \in H^{1} ( \Omega )$, as well as its continuous dependence on the boundary data $g \in L^{2} ( \partial D)$. Using a similar technique to derive \eqref{vf} and Green's 1st Theorem, it holds that 
$$
0 =  -\int_{\partial \Omega} f\overline{v} \, \text{d}s + \int_{\partial D}  (( \sigma \partial_{\nu} u)^{+} - (\sigma \partial_{\nu}  u)^{-}) \overline{v} \, \text{d}s -\int_{\partial D} u ((\sigma \partial_{\nu} \overline{v})^{+} - (\sigma \partial_{\nu}  \overline{v})^{-}) \, \text{d}s,
$$
where we used the Robin boundary condition $\sigma\partial_{\nu}v +v = 0$ on $\partial \Omega$. Using the boundary condition on $\partial D$ for $u$, we have that 
$$ \int_{\partial\Omega} f \overline{v} \, \text{d}s = \int_{\partial D} u \big( \gamma \overline{v} - [\![ \sigma \partial_\nu \overline{v} ]\!]  \big) \, \text{d}s.$$ 
From the boundary condition on $\partial D$ for $v$, we have that $$ \int_{\partial D} u \big( \gamma \overline{v} - [\![\sigma \partial_\nu \overline{v} ]\!]  \big) \, \text{d}s = \int_{\partial D} u \overline{g} \, \text{d}s.$$ 
Thus, we have that $$ (Sf , g)_{L^{2} ( \partial D)} = \int_{\partial D} u \overline{g} \, \text{d}s = \int_{\partial \Omega} f \partial_{\nu} \overline{v}\, \text{d}s = \langle f , S^{*} g \rangle_{\partial \Omega}$$
for all $f \in H^{-1/2} (\partial \Omega)$ and $g \in L^{2} (\partial D)$ and thus $S^{*} g = v\big|_{\partial \Omega}$.
\end{proof}

We conclude this section with a final preliminary result needed to consider the LSM and RFM.

\begin{theorem}\label{thm:S-dense-range}
   The solution operator $S$ has dense range.
\end{theorem}

\begin{proof}
    We prove this by equivalently showing that the adjoint $S^*$ as defined in \eqref{adjoint} is injective. To this end, suppose that $S^* g = 0 \rvert_{\partial \Omega}$, which implies that $v = 0$ on $\partial \Omega$.  
    Observe that $v + \sigma\partial_{\nu} v = 0$ on $\partial \Omega$. Since $v = 0$ on $\partial \Omega$ and $\sigma \neq 0$, it follows that $\partial_{\nu} v = 0$.  As $-\nabla \cdot \sigma \nabla v = 0$ in $\Omega \setminus \partial D$, Holmgren's theorem implies that $v = 0$ in $\Omega \setminus \partial D$. On $\partial D$, we have $(\sigma \partial_{\nu} v)^{+} - (\sigma \partial_{\nu} v)^{-} = \gamma v - g.$ Since $v = 0$ in $\Omega \setminus D$, we have $v = 0$ on $\partial D$ and $\partial_{\nu} v^+ = 0$. Therefore, the above equation reduces to $-(\sigma \partial_{\nu} v)^{-} = -g \quad \text{on } \partial D. $ Furthermore, we have $-\nabla \cdot \sigma \nabla v = 0$ in $D$ with zero Dirichlet data on $\partial D$. Thus, $v = 0$ in $D$, which implies $\partial_{\nu} v^- = 0$ on $\partial D$. Hence, $g = 0$ on $\partial D$ proving the claim. 
\end{proof}
With the mapping properties of $G$ and $S$ now established, including injectivity, compactness of their adjoint operators, and density of the associated ranges, we have assembled the analytic components required to connect the boundary data $(M-M_0)f$ to the unknown interior region $D$. The factorization \eqref{eqn:GSf} demonstrates that the measured voltage gap encodes the effect of a jump in current across $\partial D$. This structure will serve as the foundation for the qualitative reconstruction methods developed below. We now turn to the first of these, the LSM, which characterizes the inclusion $D$ by examining the solvability properties of an operator equation involving $(M-M_0)$.

\subsection{\textbf{Linear Sampling Method}}\label{LSM}
The analytical properties derived in the previous section allow us to study how the range of the data operator $(M-M_0)$ reflects the geometry of the interior region $D$. The LSM exploits this connection by examining the approximate solvability of an operator equation involving $(M-M_0)$ and the soon-to-be defined associated Robin Green’s function. The factorization \eqref{eqn:GSf} shows that the data operator maps boundary inputs on $\partial \Omega$ to Dirichlet traces arising from jumps across the unknown interface $\partial D$. This representation allows us to detect whether a sampling point lies inside $D$ by analyzing the approximate range of $(M-M_0)$. We now develop this characterization.

\begin{theorem}\label{thm:dataop_prop}
    The data operator $(M-M_0)$ given by \eqref{op:M-M0} has the factorization 
    $$(M-M_0) = GS,$$
    where $G$ and $S$ are defined in \eqref{op:G} and  \eqref{op:S}, respectively. Moreover, the operator $(M-M_0)$ is compact, injective, and with dense range.
\end{theorem}
\begin{proof}
    The factorization result was established in \eqref{eqn:GSf}. Compactness follows from the fact that $G^*$, as described in Theorem \ref{CompactG}, is compact. Injectivity is given by the fact that both $G$ and $S$ are injective as shown in Theorem \ref{thm:G-injective} and Theorem \ref{injectiveS}, respectively. Lastly, $(M-M_0)$ has dense range since $S$ has dense range, as shown in Theorem \ref{thm:S-dense-range}, and $G$ is bounded.

\end{proof}

We now introduce the Robin Green's function for the elliptic operator in the known background domain $\Omega$. For any fixed sampling point $z\in\Omega$, we define $\mathbb{G}(\cdot,z)$ as the unique solution of
\begin{equation}\label{eqn:RobinGreen}
-\nabla\cdot\sigma  \nabla \mathbb{G}(\cdot,z) = \delta(\cdot-z) \quad \text{in} \quad \Omega \quad \text{with} \quad \sigma \partial_{\nu}\mathbb{G}(\cdot,z)+\mathbb{G}(\cdot,z) = 0\quad \text{on}\quad \partial  \Omega,
\end{equation}
where $\delta(\cdot-z)$ denotes the Dirac delta at $z$. Its trace $\mathbb{G}(\cdot,z)\big|_{\partial\Omega}$ defines a function for each fixed $z\in \Omega$, and will serve as test data. The following is an instrumental result that connects the above Green's function on $\partial \Omega$ to the inclusion $D$.

\begin{theorem}
    Let the StD operator $G$ be as defined in \eqref{op:G}. Then, 
    $$\mathbb{G}(\cdot , z) \big \rvert_{\partial \Omega} \in Range(G) \quad \text{if and only if} \quad z \in D.$$
\end{theorem}
\begin{proof} $\Leftarrow$)
    We first assume that $z\in D$ and we note that $\mathbb{G}(\cdot,z)\in H^1(\Omega\setminus\overline{D})$ satisfies  
    \begin{equation*}
    -\nabla\cdot\sigma  \nabla \mathbb{G}(\cdot,z)  = 0 \quad \text{in} \quad \Omega\setminus\overline{D} \quad \text{with} \quad \sigma \partial_{\nu}\mathbb{G}(\cdot,z)+\mathbb{G}(\cdot,z) = 0\quad \text{on}\quad \partial  \Omega.
    \end{equation*}
    Now, let $\mu\in H^1(D)$ solve 
    $$ - \nabla\cdot\sigma\nabla \mu = 0 \enspace \text{in} \enspace D \quad \text{with} \quad \mu  = \mathbb{G} ( \cdot \, , z) \quad \text{on}\quad \partial D,$$ 
    and define $v_z$ such that 
 \[ v_z = \begin{cases} 
          \mathbb{G}( \cdot \, , z), & \text{in} \enspace \Omega \backslash \overline{D},\\
          \mu, & \text{in} \enspace D.
       \end{cases}
    \]
Then, we have that $v_z$ solves
\begin{align*}
- \nabla\cdot\sigma\nabla &v_z = 0 \quad \mbox{in} \quad \Omega \backslash \partial D \quad \text{with} \quad \sigma \partial_{\nu}v_z +v_z = 0\quad \mbox{on}\quad \partial \Omega.
\end{align*}
Consider $$h=\frac{[\![ \sigma \partial_\nu v_z ]\!]}{\gamma} = \frac{( \sigma\partial_{\nu} v_z)^{+}}{\gamma} - \frac{(\sigma \partial_{\nu} v_z)^{-}}{\gamma}=\frac{ (\sigma\partial_{\nu}\mathbb{G}(\cdot,z))^{+}}{\gamma} - \frac{(\sigma\partial_{\nu}\mu)^{-}}{\gamma}.$$
Since $z\in D$, then $\mathbb{G}(\cdot,z)\in H^2(\Omega\setminus D)$ by interior elliptic regularity (see \cite{evans}). Therefore, the interior trace satisfies $\mathbb{G}(\cdot,z)|_{\partial D}^+\in H^{3/2}(\partial D).$ Moreover, $\mathbb{G}(\cdot,z)$ is continuous across $\partial D$, so $\mathbb{G}(\cdot,z)^+=\mathbb{G}(\cdot,z)^-$. Since $\mu\in H^2(D),$ its interior normal derivative satisfies $\frac{(\sigma\partial_{\nu}\mu)^{-}}{\gamma}\in H^{1/2}(\partial D)\subset L^{2}(\partial D).$ Consequently, $h\in L^2(\partial D)$, and for every $z\in D$ we obtain $Gh=\mathbb{G}(\cdot,z)|_{\partial \Omega}$.

$\Rightarrow$) Assume $z\in \Omega\setminus\overline{D}$ and suppose, toward a contradiction, that $\mathbb{G}(\cdot,z)|_{\partial \Omega}\in \text{Range}(G)$. Thus, there exists $h_z\in L^2(\partial D)$ such that $Gh_z=\mathbb{G}(\cdot, z)|_{\partial \Omega}$. By definition of $G$, there exist $v_z\in H^1(\Omega)$ solving
\begin{align*}
- \nabla\cdot\sigma\nabla &v_z = 0 \quad \text{in} \quad \Omega \backslash \partial D \quad \text{with} \quad \sigma \partial_{\nu}v_z +v_z = 0\quad \text{on}\quad \partial  \Omega.\\
&\text{and}\quad [\![ \sigma \partial_\nu v_z ]\!]=\gamma h_z \quad \text{on}\quad \partial D.
\end{align*}
In particular, since $Gh_z =v_z|_{\partial \Omega}$, we have $v_z|_{\partial \Omega}=\mathbb{G}(\cdot, z)|_{\partial \Omega}$ and on $\Omega\setminus\overline{D}$ we have 
\begin{align*}
- \nabla\cdot\sigma\nabla &v_z = 0 \quad \text{in} \quad \Omega \backslash \partial D \quad \text{with} \quad \sigma \partial_{\nu}v_z +v_z = 0\quad \text{on}\quad \partial  \Omega.\\
&\text{and}\quad v_z=\mathbb{G}(\cdot, z)\quad \text{on}\quad \partial D.
\end{align*}
Set $W_z = v_z - \mathbb{G}(\cdot , z )$ and observe that since $z\notin D$, the Green's function satisfies $$- \nabla\cdot\sigma\nabla \mathbb{G}(\cdot , z )=0 \mbox{ in } \Omega \backslash ( \overline{D} \cup \{z\} ).$$ Hence, 
$$ - \nabla\cdot\sigma\nabla W_z = 0 \quad \text{in} \quad \Omega \backslash ( \overline{D} \cup \{z\} ) \quad \text{with} \quad W_z \big \rvert_{\partial \Omega} = 0 \quad \text{and} \quad \partial_{\nu} W_z \big|_{\partial \Omega} = 0.$$ 
Because both the Dirichlet and Neumann traces of $W_z$ vanish on the boundary of $D$, $\partial D$, Holmgren's Theorem implies that $W_z = 0$ in $\Omega \backslash (\overline{D} \cup \{z\})$. Therefore, $v_z = \mathbb{G} (\cdot , z )$ in $\Omega \backslash (\overline{D} \cup \{z\})$. By interior elliptic regularity, $v_z$ is continuous at $z \in \Omega \backslash \overline{D}$, but $\mathbb{G} ( \cdot ,z)$ has a singularity at $z$. This proves the claim by contradiction, due to the fact that  
$$| v_z (x)| < \infty \quad \text{ whereas}  \quad | \mathbb{G} (x,z) | \rightarrow \infty \quad \text{as} \quad x \rightarrow z.$$
\end{proof}
By the preceding analysis, we can now state the main result of the LSM, which characterizes the inclusion $D$ via an approximate solution of the operator equation 
\begin{equation}\label{eq:cont-system}
(M - M_0) f_z = \mathbb{G}(\cdot, z)\big|_{\partial \Omega}.
\end{equation}  

\begin{theorem}\label{mainlsm}
    Let the data operator $(M-M_0)$ be as defined in \eqref{op:M-M0}. For any $z \in \Omega$, if there exists a sequence $\{f^z_{\epsilon}\}_{\epsilon >0} \in H^{-1/2}(\partial \Omega)$ such that
    $$\| (M-M_0) f^z_{\epsilon} - \mathbb{G}(\cdot , z) \rvert_{\partial \Omega}\|_{H^{1/2}(\partial \Omega)} \rightarrow 0 \quad \text{as} \quad \epsilon \rightarrow 0,$$
    then we must have that $\| f^z_{\epsilon}\|_{H^{-1/2}(\partial \Omega)} \rightarrow \infty$ as $\epsilon \rightarrow 0$ for all $z \notin D$.
\end{theorem}
\begin{proof}
    Since $(M-M_0)$ has dense range in $H^{1/2}(\partial \Omega)$ (by Theorem \ref{CompactG}), for any $z\in \Omega$, there exists an approximating sequence $\{f_\epsilon^z\}_{\epsilon>0}$ such that $(M-M_0)f_\epsilon^z$ converges in norm to $\mathbb{G}(\cdot,z)|_{\partial \Omega}$. Assume by contradiction that $\|f_\epsilon^z\|_{H^{-1/2}(\partial \Omega)}$ remains bounded as $\epsilon\to 0.$ Then, there exists a subsequence, $\{f_{\epsilon_k}^z\}$, that is weakly convergent such that $f_{\epsilon_k}^z\rightharpoonup f_{\epsilon_0}^z$ as $\epsilon\to 0.$ Since $(M-M_0)$ is compact, then as $\epsilon\to 0$, we have  $$(M-M_0)f_{\epsilon_k}^z\rightarrow (M-M_0)f_{\epsilon_0}^z 
\quad \text{implying that} \quad  (M-M_0)f_{\epsilon_0}^z=\mathbb{G}(\cdot,z)|_{\partial \Omega}.$$
This says that $\mathbb{G}(\cdot,z)|_{\partial \Omega}\in\text{Range}(G)$, which is equivalent to $z\in D$. This contradicts our previous Theorem, proving the result, which shows that $\|f_\epsilon^z\|_{H^{-1/2}(\partial \Omega)}\to \infty$ as $\epsilon\to 0$ for all $z\notin D.$
\end{proof}

Since the operator $(M - M_0)$ is compact, the problem~\eqref{eq:cont-system} is ill-posed. However, since $(M - M_0)$ possesses a dense range, one can obtain approximate solutions through a suitable regularization procedure (see, e.g., \cite{kircharti}). Consequently, to reconstruct $D$ via the LSM, it is indeed viable to approximate the solution to \eqref{eq:cont-system} via certain regularization methods.  Thus, we define the imaging functional  
\begin{equation}\label{lsm-functional}
W_{\text{LSM}}(z) = \frac{1}{\| f^z_{\epsilon} \|_{H^{-1/2}(\partial \Omega)}},
\end{equation}
where $\epsilon > 0$ is a specific regularization parameter. Note that \eqref{lsm-functional} nearly vanishes for sampling points $z \notin D$, i.e., $W_{\text{LSM}}(z)\approx 0$. However, for $z \in D$, no lower bound for $W_{\text{LSM}}(z)$ can be guaranteed. For this reason, in the following section, we turn our attention to the RFM to derive an additional imaging functional with a lower bound.  

\subsection{\textbf{Regularized Factorization Method}}\label{RegFM}
In this section, we consider using the RFM to recover the region $D$, i.e., the interior of the corroded interface $\partial D$. The theory was developed in \cite{harris1} and has been recently adapted to recover elastic subdomains (see \cite{{granados-marzuola-rodriguez1}}). Even though we have shown that the LSM recovers $D$ in Theorem~\ref{mainlsm}, our result does not imply that the corresponding imaging functional~\eqref{lsm-functional} is bounded below for sampling points $z \in D$. Thus, we consider the factorization method since it provides an exact characterization of $D$ via the spectral decomposition of $(M-M_0)$. This will require a more detailed factorization of the data operator than the one demonstrated in Theorem~\ref{thm:dataop_prop}, namely $(M-M_0) = GS$. In fact, we will factorize the operator $G$ involving the adjoint of $S$ as defined in Theorem~\ref{adjoint}. This will yield a symmetric factorization of the data operator. 

Recall that $S\colon H^{-1/2}(\partial \Omega) \rightarrow L^{2}(\partial D)$ and $S^{*}\colon L^{2}(\partial D) \rightarrow H^{1/2}(\partial \Omega)$. To formally derive a symmetric factorization of the data operator, it is necessary to introduce an intermediate operator $T$. Recall that $w$ is the unique solution to \eqref{w}, which indicates that $-\nabla \cdot \sigma \nabla w = 0$ in $\Omega \setminus \partial D$ and 
$$ [\![\sigma \partial_\nu w ]\!] \big|_{\partial D} = \gamma w\big|_{\partial D} +  \gamma \big[ h - w\big|_{\partial D} \big]$$
by appealing to the Robin transmission condition. Therefore, we have 
$$\sigma \partial_{\nu} w \big|_{\partial \Omega} = Gh \quad \text{as well as} \quad  \sigma \partial_{\nu} w \big|_{\partial \Omega} =S^{*} \gamma \big[ h - w \big|_{\partial D} \big]$$
by the well-posedness of \eqref{adv} and Theorem \ref{adjoint}. The statements above motivate the definition of the operator 
\begin{equation}\label{Toperator}
T\colon L^2 ( \partial D ) \rightarrow L^2 ( \partial D ) \quad \text{given by} \quad Th =  \gamma \big[ h - w|_{\partial D} \big],
\end{equation}
where $w$ satisfies \eqref{w}. As a consequence of its well-posedness, one can show that $T$ is a bounded linear operator. We had already established that $(M - M_0) = GS$. However, we have now factorized $G = S^{*}T$. This yields the following result.

\begin{theorem}\label{factorization} 
The data operator $(M - M_0)\colon H^{-1/2} ( \partial \Omega) \rightarrow H^{1/2} ( \partial \Omega )$ defined in \eqref{op:M-M0} has the symmetric  factorization $(M - M_0)= S^{*} TS$. 
\end{theorem}

To apply Theorem 2.3 from \cite{harris1} to solve the inverse problem of recovering $D$ from the data operator, it remains to establish that $(M-M_0)$ is positive and that the operator $T$ is coercive. We aim to describe the interior region $D$ in terms of the range of $S^*$, as it will enable the reconstruction of $D$ from the measured data $Mf$ and the known data $M_0f$ on the accessible outer boundary. The next two results provide the groundwork for analyzing key features of the data operator through the symmetric factorization introduced in the preceding theorem. We begin by demonstrating that $T$ is coercive.

\begin{theorem}\label{tcoercive}
The operator $T$ defined in \eqref{Toperator} is coercive on $L^2(\partial D)$.
\end{theorem}
\begin{proof} Using the Robin transmission condition on $\partial D$ in equation \eqref{w}, we have that $$ (Th, h)_{L^2 ( \partial D) } = \int_{\partial D} \gamma (h - w) \overline{h} \, \text{d}s =  \int_{\partial D} \gamma |h|^2 \, \text{d}s - \int_{\partial D} w \, [\![\partial_\nu \sigma \overline{w} ]\!] \, \text{d}s, $$
where $[\![\sigma\partial_\nu w ]\!]\big|_{\partial D} = \gamma h.$
Following a similar technique used to derive \eqref{vf}, we have that 
$$ \int_{\Omega \backslash \overline{D} }\sigma\nabla w\cdot \nabla\overline{w}\, \text{d}x =  \int_{\partial \Omega} w  \sigma\partial_{\nu} \overline{w} \, \text{d}s- \int_{\partial D} w  (\sigma\partial_{\nu} \overline{w})^{+} \, \text{d}s$$
$$\text{and}\quad  \int_{D}\sigma\nabla w\cdot \nabla\overline{w}\, \text{d}x= \int_{\partial D} w (\sigma \partial_{\nu} \overline{w})^{-} \, \text{d}s.$$
Adding both equations above, using the boundary condition on $\partial D$ and that $\sigma \partial_{\nu}w+w=0$ on $\partial \Omega$ yields
 $$ \int_{\Omega} |\nabla w|^2 \, \text{d}x =-\int_{\partial \Omega}|w|^2 \, \text{d}s- \int_{\partial D} w [\![\sigma \partial_\nu \overline{w} ]\!] \, \text{d}s.$$ Therefore, we have that 
 $$(Th, h)_{L^2 ( \partial D) } = \int_{\partial D} \gamma |h|^2 \, \text{d}s + \int_{\partial \Omega} |w |^{2} \, \text{d}s+ \int_{\Omega} \sigma\nabla w\cdot \nabla\overline{w} \, \text{d}x\geq \gamma_{\text{min}} \int_{\partial D}  |h|^2 \, \text{d}s$$ 
which proves the claim.
\end{proof}
From this coercivity result, we can show that $(M-M_0)$ is positive. Indeed
\begin{align*}
 \langle (M- M_0) f , f\rangle_{\partial \Omega} &=  \langle  S^*TS f , f\rangle_{\partial \Omega}\\
    &= \big(TS f , Sf \big)_{L^2(\partial D)} \\
    &\geq \norm{ S f }^2_{L^2(\partial D)} \\
    &=\norm{u}^2_{L^2(\partial D)},
\end{align*} 
where we have used that the operator $T$ is coercive and  $S \colon H^{-1/2} ( \partial \Omega ) \rightarrow L^{2} (\partial D )$ is given by $Sf = u \rvert_{\partial D}$. Therefore, there exists an operator $Q\colon H^{-1/2} (\partial \Omega) \rightarrow L^{2}(\partial \Omega)$ such that $(M-M_0) = Q^{*}Q$, i.e., $(M-M_0)^{1/2} = Q^{*}$. By the above results and Theorem \ref{thm:dataop_prop}, the operator $(M-M_0)$ satisfies the hypotheses of Theorem 2.3 from \cite{harris1}. Specifically, $Range(Q^{*}) = Range(S^{*})$ and that we have the equivalence
\[
\ell \in \text{Range}(S^{*}) \quad \text{if and only if} \quad \liminf_{\alpha \to 0} \langle (M - M_0) f_{\alpha} , f_{\alpha} \rangle_{\partial \Omega} < \infty,
\]
where $f_{\alpha}$ denotes the regularized solution to the equation $(M - M_0) f = \ell$. Given that the data operator $(M-M_0)$ is compact, injective, and has dense range, it is amenable to standard regularization techniques such as Tikhonov or Spectral cut-off. Nevertheless, it is still necessary to relate the interior region $D$ to the range of the adjoint operator $S^{*}$. To achieve this, we reconsider the Robin Green's function for the conductivity operator on the known domain $\Omega$, denoted by $\mathbb{G}(\cdot, z)$ for any fixed point $z \in \Omega$. The key idea of the forthcoming result is that due to the singularity of $\mathbb{G}(\cdot, z)$ at $z$, the trace of Robin Green's function on $\partial \Omega$ fails to lie in the range of $S^{*}$ unless the sampling point $z$ is contained within the interior region $D$.

\begin{theorem} \label{greenchar}
The operator $S^{*}$ is such that for any $z \in \Omega$, 
$$ \mathbb{G} (\cdot , z) \big|_{\partial \Omega} \in Range(S^{*}) \quad \text{if and only if} \quad z \in D.$$
\end{theorem}

\begin{proof} 
$\Rightarrow$) To prove the claim, assume $z \in \Omega \backslash \overline{D}$. Suppose by contradiction that there exists $g_z \in L^2 ( \partial D )$ such that $S^{*} g_z = \mathbb{G} ( \cdot , z) \big \rvert_{\partial \Omega}$. This implies that there exists $\, v_z \in H^{1} ( \Omega )$ such that 
\begin{align*}
- \nabla\cdot\sigma\nabla &v_z = 0 \quad \text{in} \quad \Omega \backslash \partial D \quad \text{with} \quad \sigma \partial_{\nu}v_z +v_z = 0\quad \text{on}\quad \partial  \Omega\\
&\text{and}\quad \quad [\![\sigma\partial_\nu v ]\!]\big|_{\partial D} = \gamma v_z- g_z.
 \end{align*}
Furthermore, we have that $S^*g_z=\mathbb{G} ( \cdot , z) \big \rvert_{\partial \Omega}$ and on $\Omega\backslash D$, 
\begin{align*}
- \nabla\cdot\sigma\nabla &v_z = 0 \quad \text{in} \quad \Omega \backslash \partial D \quad \text{with} \quad \sigma \partial_{\nu}v_z +v_z = 0\quad \text{on}\quad \partial  \Omega\\
&\text{and} \quad v_z=\mathbb{G} ( \cdot , z) \quad \text{on} \quad \partial \Omega.
 \end{align*}
Now, we will define $W_z = v_z - \mathbb{G}(\cdot , z )$ and note that since $z\in \Omega\backslash D$, we have that \linebreak $- \nabla\cdot\sigma\nabla \mathbb{G}(\cdot , z )=0$ in $\Omega \backslash ( \overline{D} \cup \{z\} ).$ Thus, we have that $W_z$ satisfies the following 
$$ - \nabla\cdot\sigma\nabla W_z = 0 \quad \text{in} \quad \Omega \setminus ( \overline{D} \cup \{z\} ) \quad \text{with} \quad W_z \big \rvert_{\partial \Omega} = 0 \quad \text{and} \quad \partial_{\nu} W_z \big|_{\partial \Omega} = 0.$$ 
By Holmgren's Theorem, we conclude that $W_z = 0$ in $\Omega \setminus (\overline{D} \cup \{z\})$. That is, $v_z = \mathbb{G} (\cdot , z )$ in $\Omega \setminus (\overline{D} \cup \{z\})$. By interior elliptic regularity, $v_z$ is continuous at $z \in \Omega \setminus \overline{D}$, but $\mathbb{G} ( \cdot ,z)$ has a singularity at $z$. This proves the claim by contradiction, due to the fact that  
$$| v_z (x)| < \infty \quad \text{ whereas}  \quad | \mathbb{G} (x,z) | \rightarrow \infty \quad \text{as} \quad x \rightarrow z.$$

$\Leftarrow$) Conversely, assume that $z \in D$. Our goal is to construct $g_z \in L^2(\partial D)$ satisfying $S^*g_z=\mathbb{G} ( \cdot , z) \big \rvert_{\partial \Omega}$. Let $\mu \in H^1 (D)$ be the solution to the following problem in $D$
$$ - \nabla\cdot\sigma\nabla \mu = 0 \enspace \text{in} \enspace D \quad \text{with} \quad \mu  = \mathbb{G} ( \cdot \, , z) \quad \text{on}\quad \partial D.$$
Now, define $v_z$ such that 
 \[ v_z = \begin{cases} 
          \mathbb{G}( \cdot \, , z) & \text{in} \enspace \Omega \backslash \overline{D}\\
          \mu & \text{in} \enspace D.
       \end{cases}
    \]
We will show that $v_z$ satisfies all of the conditions imposed by \thmref{adjoint}. 
By definition, we see that $v_z$ is harmonic in $\Omega \backslash \partial D$, $v_z \in H^1(\Omega)$ since there is no jump in the trace across $\partial D$, and $\sigma v_z+v_z=0$ on $\partial \Omega$ as $\partial_{\nu}\mathbb{G}(\cdot,z)+\mathbb{G}(\cdot,z) = 0$ on $\partial \Omega$. Now, we prove that 
$$g_z =  \gamma v_z \big|_{\partial D}- [\![\partial_\nu v_z ]\!]  \big|_{\partial D} $$
 is in $L^2 ( \partial D)$. To this end, let $v_z\in H^1(D)$ with $v_z\big|_{\partial D}\in H^{1/2}\subseteq L^2(\partial D)$. Since \linebreak $\sigma\in L^{\infty}(\partial D)$, it follows that $\sigma v_z\in L^2(\partial D).$ We now consider 
 $$ [\![\sigma \partial_\nu v_z ]\!] = \Big(\sigma \partial_{\nu} \mathbb{G} ( \cdot , z) \Big)^+ - \Big(\sigma \partial_{\nu} \mu \Big) ^{-}\quad \text{on}\quad \partial D .$$
Since  $z \in D$, we have $\mathbb{G} ( \cdot , z ) \in H^{2} ( \Omega \setminus  \overline{D})$. Therefore,  
$$\partial_{\nu}\mathbb{G} ( \cdot , z)\big|^+_{\partial D} \in  H^{1/2} ( \partial D)  \subseteq L^2(\partial D),$$
as we have that $\mu |_{\partial D}= \mathbb{G}(\cdot, z) \rvert_{\partial D} \in H^2(\Omega\setminus D) $
by appealing to interior elliptic regularity (see \cite{evans}). By the Neumann Trace Theorem, we  obtain that 
$$ [\![\sigma \partial_\nu v_z ]\!]  \big|_{\partial D} \in H^{1/2} ( \partial D) \subseteq  L^2 ( \partial D).$$
In addition, it is clear that $ \gamma v_z \rvert_{\partial D}\in L^2 ( \partial D)$. Consequently, we conclude that $g_z \in L^2 ( \partial D)$. Finally, by appealing to Theorem \ref{adjoint} we have $S^* g_z = \mathbb{G}(\cdot  \, , z )|_{\partial \Omega}$, which proves the claim. 
\end{proof}

By applying Theorem 2.3 in \cite{harris1}, we now present the main result of the RFM aimed at recovering an unknown interior region $D$ from the data operator $(M-M_0)$. 

\begin{theorem}\label{factmain}
The data operator $(M - M_0)\colon H^{-1/2}(\partial \Omega) \to H^{1/2}(\partial \Omega)$ uniquely characterizes the domain $D$. Moreover, for any point $z \in \Omega$,
\[
z \in D \quad \text{if and only if} \quad \liminf_{\alpha \to 0} \langle (M - M_0) f_{\alpha}^{z} , f_{\alpha}^{z} \rangle_{\partial \Omega} < \infty,
\]
where $f_{\alpha}^{z}$ is a regularized solution to the equation $(M - M_0) f^{z} = \mathbb{G}(\cdot, z)\big|_{\partial \Omega}$.
\end{theorem}

This theorem provides another qualitative approach to shape recovery of (possibly multiple) regions from boundary data. Furthermore, since $(M-M_0)$ is compact, the result of Theorem \ref{factmain} can be reformulated as 
\[ z \in D \quad \text{if and only if} \quad \sum_{n=1}^{\infty} \frac{1}{s_n} \big \rvert \langle \mathbb{G} (\cdot , z)\rvert_{\partial \Omega} , \ell_n \rangle_{\partial D} \big \rvert^2 < \infty,  \]
following Picard's criteria (see Theorem A.58 in \cite{kircharti}). Here $\{s_n ; \ell_n ; x_n\}_{n=1}^{\infty} \in \mathbb{R}_{+} \times H^{-1/2}(\partial \Omega) \times H^{1/2}(\partial \Omega)$ is the singular value decomposition of $(M-M_0)$. This demonstrates that every sampling point $z \in \Omega$ yields a computable infinite series, and is stronger than Theorem \ref{mainlsm} as it provides an equivalence between the sampling point being contained in the region of interest and its corresponding series being finite. Moreover, this equivalence implies that the data operator $(M - M_0)$ uniquely determines the subregion~$D$. For numerical reconstructions of $D$, we define the imaging functional
\begin{equation}\label{indi2}
W_{RFM} (z) = \Bigg[ \sum_{n=1}^{\infty} \frac{\phi^2_{\alpha}(s_n)}{s_n} \big\rvert  \langle \mathbb{G} (\cdot , z)\rvert_{\partial \Omega} , \ell_n \rangle_{\partial D}\big \rvert^2 \Bigg]^{-1},
\end{equation}
where $\phi_{\alpha}$ represents a filter function corresponding to a certain regularization scheme, as shown in \cite{harris1}. We describe this in more detail in the next section. Note that this imaging functional is positive only when $z \in D$. We remark that since $(M-M_0)$ is compact, its singular values $s_n$ tend rapidly to zero, which may cause numerical instabilities. The filter function $\phi_{\alpha}$ is therefore introduced to regularize the contribution of small singular values. 

\section{Numerical Results}\label{section: numerics}
In this section, we present numerical examples to validate the LSM introduced in Section~\ref{LSM} and the RFM introduced in Section~\ref{RegFM}. All our numerical examples have been done in  \texttt{MATLAB} 2022b. Our implementations of the methods described in this paper will be available at \url{https://github.com/malenaespanol/ReMeShapeEIT}.

For simplicity, we will consider the problem where $\Omega\subset \mathbb{R}^2$ is the unit disk. Thus, the trace spaces $H^{\pm 1/2}(\partial \Omega)$ can be identified with $H^{\pm 1/2}_{\text{per}}([0,2\pi])$. Furthermore, we assume that the conductivity $\sigma$ is constant and positive. In order to apply Theorems \ref{mainlsm} and \ref{factmain}, we need the Robin Green's function on the unit disk. As shown in~\cite{RtDGreens}, taking the polar coordinates \(x=(r,\varphi)\), \(z=(\rho,\theta)\), 
and denoting \(\lambda = \theta - \varphi \), the Robin Green’s function \(\mathbb{G}(x,z)\) has the form
\[\mathbb{G}(x,z) = \mathbb{G}_D(x,z) + \frac{1}{2\pi}\int_{0}^{1} s^{\frac{1}{\sigma}-1}\, 
P(r\rho s, \lambda)\, ds.
\]
Here \(\mathbb{G}_D(x,z)\) is the Dirichlet Green’s function and $P(t,\lambda) = \frac{1-t^2}{1-2t\cos\lambda + t^2}$ is the Poisson kernel. Note that \(\mathbb{G}_D(x,z)\) evaluated on the unit circle is zero. Thus, for all $z \in \Omega$, \[ \mathbb{G}(x,z) \big \rvert_{x\in \partial \Omega} = \frac{1}{2\pi}\int_{0}^{1} s^{\frac{1}{\sigma}-1}\, 
P(\rho s, \lambda)\, ds. \]

To numerically solve Equation \eqref{eq:cont-system}, we denote  ${\bf A} \in \C^{N \times N}$ for $N \in \N$ as the discretized operator of $(M-M_0)$ (see Appendix~\ref{app: A} for details) and the vector ${\bf b}_z = \big [ \mathbb{G}\big ( \varphi_j ,z \big ) \big ]_{j=1}^{N}$, with $\varphi_j \in [0,2\pi)$ uniformly chosen. 

In polar coordinates, $\partial D$ is given by $\rho(\text{cos}(\varphi), \text{sin}(\varphi))$ for some constant $\rho \in (0,1)$. Using separation of variables, we have that for any $\varphi \in [0,2\pi)$, 
\[  (M-M_0) f(\varphi) = \frac{1}{2\pi} \int_{0}^{2\pi} K(\varphi , \phi) f (\phi) \, \text{d}\phi, \quad \text{where} \quad K(\varphi , \phi) = \sum_{|n|=0}^{\infty} \kappa_n e^{in \varphi}.\]
See Appendix for details on the calculation of the coefficients $\kappa_n$ for all $n \in \mathbb{Z}$. In our examples, we approximate the kernel function $K(\varphi,\phi)$ by truncating the series for $|n|=0,\hdots,10$. By having this, we then discretize the truncated integral operator using a collocation method in a 32 equally spaced grid points on $[0,2\pi)$.

In our examples, we add random noise to the matrix $\bf{A}$ such that
$$ \text{\textbf{A}}_{i,j}^{\delta} = \big [ \text{\textbf{A}}_{i,j} \big( 1 + \delta \text{\textbf{E}}_{i,j} \big) \big ]_{i,j=1}^{N} ,$$ 
where the matrix ${\bf E}$ has i.i.d. entries following a uniform distribution over $[-1,1]$, and is normalized so that $\|{\bf E}\|_{2} = 1$. The scalar $\delta$ is the relative noise level added to the data, selected so that $\|{\text{\textbf{A}}^{\delta} - \text{\textbf{A}}}\|_{2} \leq \delta \|{\text{\textbf{A}}}\|_{2}$. 
The imaging functionals defined in Theorems \ref{mainlsm} and \ref{factmain} associated with the LSM and RFM, respectively, are
$$ W_{\text{LSM}} (z) = \|\bf f^{\bf z}_\alpha\|_2^{-1} \quad \mbox{ and } \quad W_{\text{RFM}} (z) =  \sum\limits_{n=1}^{N} \frac{\phi^2_{\alpha}(s_n )}{s_n} \big|({\bf u}_n , {\bf b}_z)\big|^2,$$
where  $\bf f^{\bf z}_\alpha$ is a regularized solution of the system $\bf A^\delta \bf f^{\bf z} = \bf b_{\bf z},$
and $s_n$ and ${\bf u}_n$ denote the singular values and left singular vectors of the matrix ${\textbf{A}}^{\delta}$, respectively. The function $\phi_\alpha(t)$ denotes the filter function defined by a regularization scheme. Formally, $\{\phi_{\alpha}(t)\}_{\alpha > 0}$ is a sequence of functions where $\phi_{\alpha}(t) \colon (0, \sigma_1 ] \rightarrow \mathbb{R}_{\geq 0}$, and verify that for $0<t\leq \sigma_1$,
\begin{equation}\label{filter-functions}
\lim_{\alpha \rightarrow 0} \phi_{\alpha}(t) = 1 \quad \text{and} \quad \phi_{\alpha}(t) \leq C \quad \text{for all} \enspace \alpha >0,
\end{equation}
where $C>0$ is an absolute constant. The filter functions we use in our examples are given by 
\begin{align}\label{filters}
 \phi^{\mathrm{TIK}}_\alpha(t) =  \frac{t^2}{t^2+\alpha}, \,\, \, \textrm{ and } \,\,  \displaystyle{  \phi^{\mathrm{Spectral-cutoff}}_\alpha(t)= \left\{\begin{array}{lr} 1, &  t^2\geq \alpha,  \\
 				&  \\
 0,&  t^2 < \alpha
 \end{array} \right.}  
\end{align}
which corresponds to Tikhonov regularization and the Spectral-cutoff approach, respectively.

In addition, we consider the filtered solution yielded by applying the truncated total least squares (TTLS) approach~\cite{fierro1997regularization}. This regularization approach accounts for noise not only in the right-hand side vector $\textbf{b}_z$  but also in the matrix $\textbf{A}^\delta$, which is appropriate in our setting. Let the singular value decomposition of the augmented matrix $[\textbf{A}^\delta\,|\,\textbf{b}_z] = \bar{\bf U} \bar{\bf \Sigma} \bar{\bf V}^T,$
with singular values $\bar{\sigma}_1 \geq \cdots \geq \bar{\sigma}_{N+1}$. Let $\bar{v}_{N+1,:}$ denote the last row of $\bar{\bf V}$, and therefore $\bar{v}_{N+1,m}$ denote the $m$-th entry of it, and $\bar{v}_{N+1,k+1:N+1}$ the vector containing its last $n-k$ entries. Then, the TTLS filter function is given by
\begin{equation}\label{filter_tls}
\phi^{\mathrm{TLS}}_\alpha(t) = \sum_{m = 1}^{\alpha} \frac{\bar{v}_{N+1,m}^2}{\|\bar{v}_{N+1,k+1:N+1}\|_2^2} \frac{t^2}{ \bar{\sigma}_m^2 -t^2}.
\end{equation}

In all the experiments, the regularization parameter $\alpha$ was selected empirically based on visual inspection of the reconstructions. Developing parameter selection strategies in this setting will be left for future work. 

In the following figures, we display the normalized imaging functionals: $$\quad  W_{\text{LSM}}(z) = \frac{W_{\text{LSM}} (z)}{\|{W_{\text{LSM}} (z)}\|_{\infty}}  \qquad \mbox{ and } \quad W_{\text{RFM}}(z) =\frac{W_{\text{RFM}} (z)}{\|{W_{\text{RFM}} (z)}\|_{\infty}}.$$
Theorems \ref{mainlsm} and \ref{factmain} imply that these normalized indicators satisfy $W(z)\approx 1$ for $z \in D$, and $W(z)\approx 0$ when $ z \notin D$. 

 \begin{figure}[t]
\centering 
\includegraphics[width=1\linewidth]{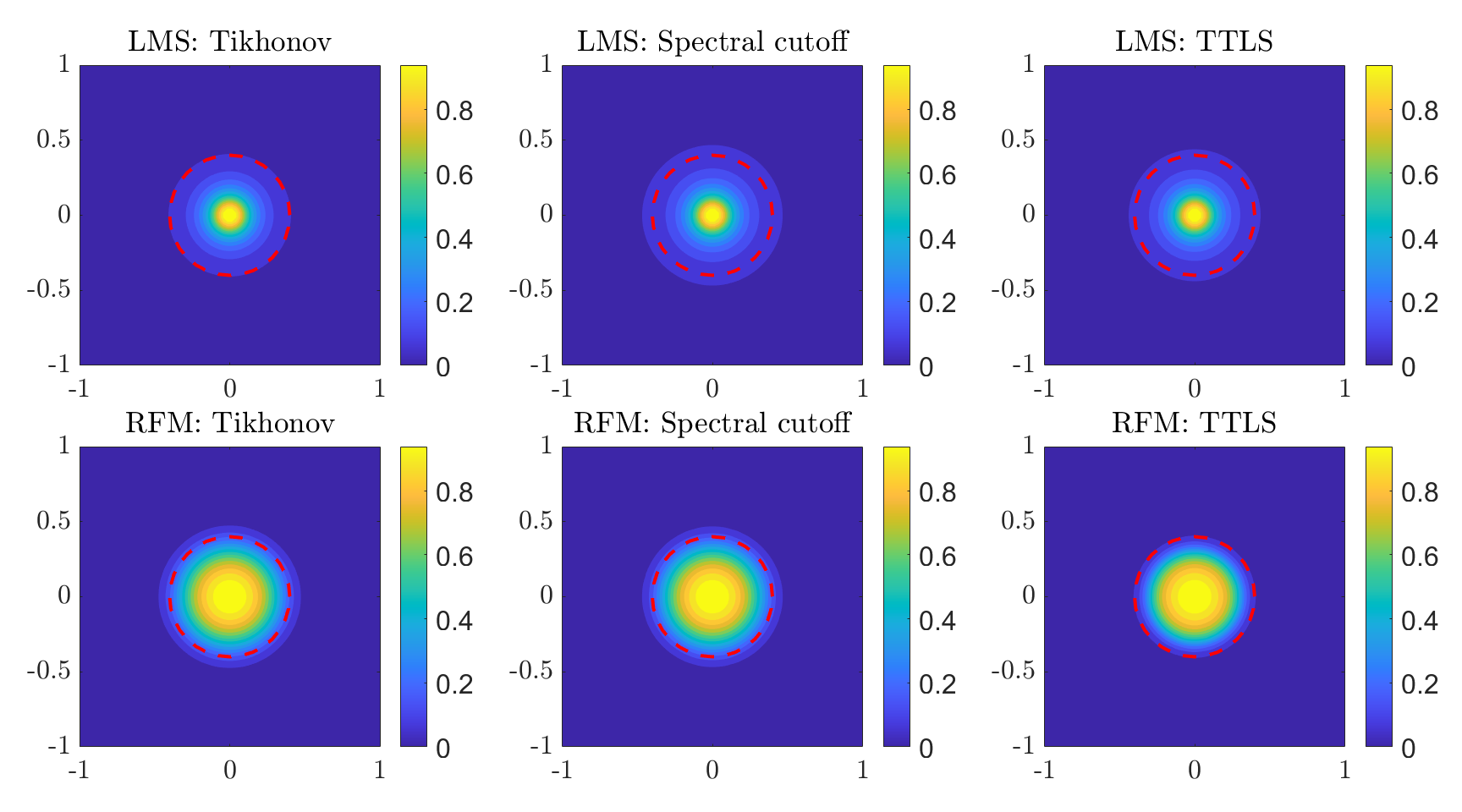}
\caption{Reconstruction of a circular region with $\rho=0.4$ marked with a dashed red line, via both LSM and RFM methods. In this case, there is no noise added in the matrix (i.e., $\delta=0$).} 
\label{fig: recon-circle-0.4-noisefree}
\end{figure}
 
Figures \ref{fig: recon-circle-0.4-noisefree} and \ref{fig: recon-circle-0.4} show the reconstructions of a circular region with radius $\rho = 0.4$ for a noise-free case ($\delta = 0$) and for data corrupted with $1 \% $ relative random noise ($\delta = 0.01$). The dashed line marks the boundary of $D$. For the noise-free case, the regularization parameters used for each approach were $\alpha_{\mathrm{TIK}}=10^{-9}$, $\alpha_{\mathrm{Spec}}=10^{-9}$,  and $\alpha_{\mathrm{TTLS}}=5$, for LSM, and $\alpha_{\mathrm{TIK}}=10^{-16}$, $\alpha_{\mathrm{Spec}}=10^{-16}$, and $\alpha_{\mathrm{TTLS}}=5$ for RFM. For the case with noise, the corresponding regularization parameters were $\alpha_{\mathrm{TIK}}=10^{-9}$, $\alpha_{\mathrm{Spec}}=10^{-9}$,  and $\alpha_{\mathrm{TTLS}}=5$ for LSM, and $\alpha_{\mathrm{TIK}}=10^{-16}$, $\alpha_{\mathrm{Spec}}=10^{-16}$, and $\alpha_{\mathrm{TTLS}}=5$ for RFM. Both methods accurately recover the circular shape in the noise-free case, with RFM reconstructions exhibiting a sharper separation between the bright interior and the blue exterior regions than LSM reconstructions. When $1\%$ noise is added, the reconstructions remain stable: the LSM indicator yields reconstructions that are nearly identical to those of the noise-free case. In contrast, the RFM indicator with Tikhonov and Spectral Cut-off regularization produces reconstructions that exhibit small peak artifacts. The newly proposed TTLS regularization, however, produces smoother reconstructions. Overall, both methods, across all regularization strategies, correctly localize the interior region, demonstrating strong robustness to small levels of noise.

\begin{figure}[t]
\centering 
\includegraphics[width=1\linewidth]{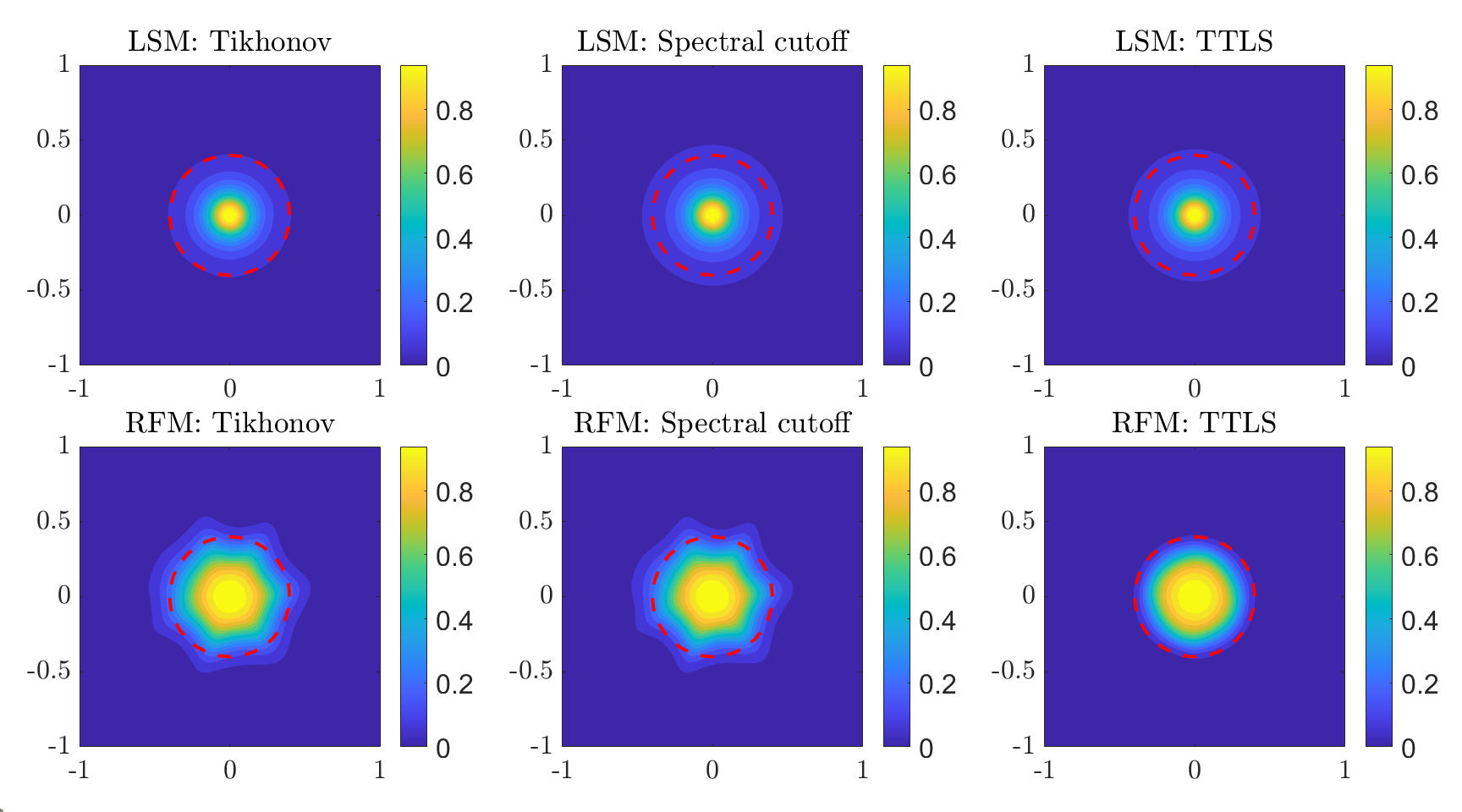}
\caption{Reconstruction of a circular region (radius $\rho = 0.4$) shown by the dashed red line, with conductivity $\sigma = 1$, using both the LSM and RFM methods. The data matrix includes $1\%$ additive noise ($\delta = 0.01$).}\label{fig: recon-circle-0.4}
\end{figure}

In Figures~\ref{fig: recon-circle-0.4-two sigmas} and \ref{fig: recon-circle-0.4-two sigmas_N64}, we again consider a circular inclusion with radius $\rho = 0.4$ and noise level $\delta = 0.01$, but now the conductivity is discontinuous: the background conductivity is $\sigma_{\text{out}} = 1$, while the interior conductivity is $\sigma_{\text{in}} = 10$. The regularization parameters used were $\alpha_{\mathrm{TIK}} = 10^{-11}$, $\alpha_{\mathrm{Spec}} = 10^{-11}$, and $\alpha_{\mathrm{TTLS}} = 10$ for the LSM reconstructions, and $\alpha_{\mathrm{TIK}} = 0$, $\alpha_{\mathrm{Spec}} = 0$, and $\alpha_{\mathrm{TTLS}} = 12$ for the RFM reconstructions. Figure~\ref{fig: recon-circle-0.4-two sigmas} uses $N = 32$ boundary sampling points, whereas Figure~\ref{fig: recon-circle-0.4-two sigmas_N64} uses $N = 64$ equiangular sampling points. In this higher-contrast setting, both methods successfully recover the correct location and shape of the inclusion. The RFM reconstruction again produces a sharper delineation of the boundary, while the LSM indicator is more diffuse inside $D$. Increasing the number of boundary sampling points from $N = 32$ to $N = 64$ noticeably improves the stability of both methods. In particular, the RFM reconstructions with Tikhonov and spectral cut-off regularization for $N = 64$ exhibit fewer spurious peaks and smoother indicator maps. These results confirm that both approaches remain reliable under strong conductivity contrasts, with the RFM together with TTLS regularization yielding the cleanest boundaries. It is noteworthy that, although our theoretical analysis assumes a single constant conductivity, this example demonstrates that both methods remain effective even when distinct conductivities are present inside and outside the inclusion.

Finally, Figure~\ref{fig: recon-circle-0.1} includes a smaller interior region with radius 
$\rho = 0.1$ and the same noise level $\delta = 0.01$ as before, for constant conductivity $\sigma = 1$.
The regularization parameters used were
$\alpha_{\mathrm{TIK}} = 10^{-13}$, $\alpha_{\mathrm{Spec}} = 10^{-15}$, and $\alpha_{\mathrm{TTLS}} = 4$ for LSM,
and $\alpha_{\mathrm{TIK}} = 0$, $\alpha_{\mathrm{Spec}} = 0$, and $\alpha_{\mathrm{TTLS}} = 5$ for RFM. 
Despite the smaller target, both methods successfully identify its location. Surprisingly, the RFM indicator (without regularization) produces better reconstruction than that obtained with TTLS regularization. 

Overall, the numerical results demonstrate that both LSM and RFM reliably reconstruct interior regions of varying size and contrast, with RFM generally producing smoother and more stable indicators under noisy data, as expected.

\begin{figure}[!t]
\centering 
\includegraphics[width=1\linewidth]{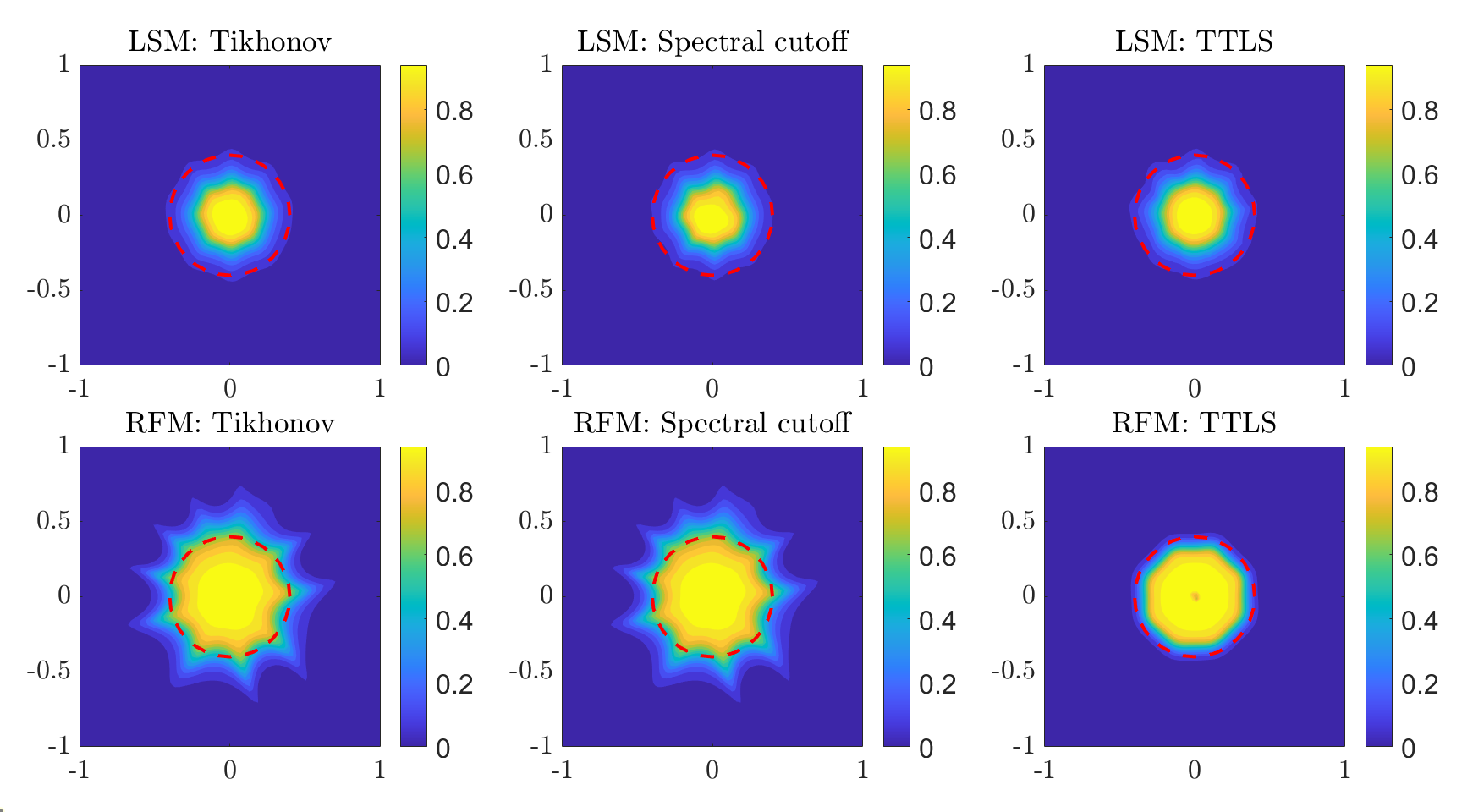}
\caption{Reconstruction of a circular region (radius $\rho = 0.4$) shown by the dashed red line, where the exterior conductivity is $\sigma_{\text{out}} = 1$ and the interior conductivity is $\sigma_{\text{in}} = 10$. The reconstruction is obtained using both the LSM and RFM methods, with $1\%$ additive noise included in the data matrix ($\delta = 0.01$). }
\label{fig: recon-circle-0.4-two sigmas}
\end{figure}

\begin{figure}[!t]
\centering 
\includegraphics[width=1\linewidth]{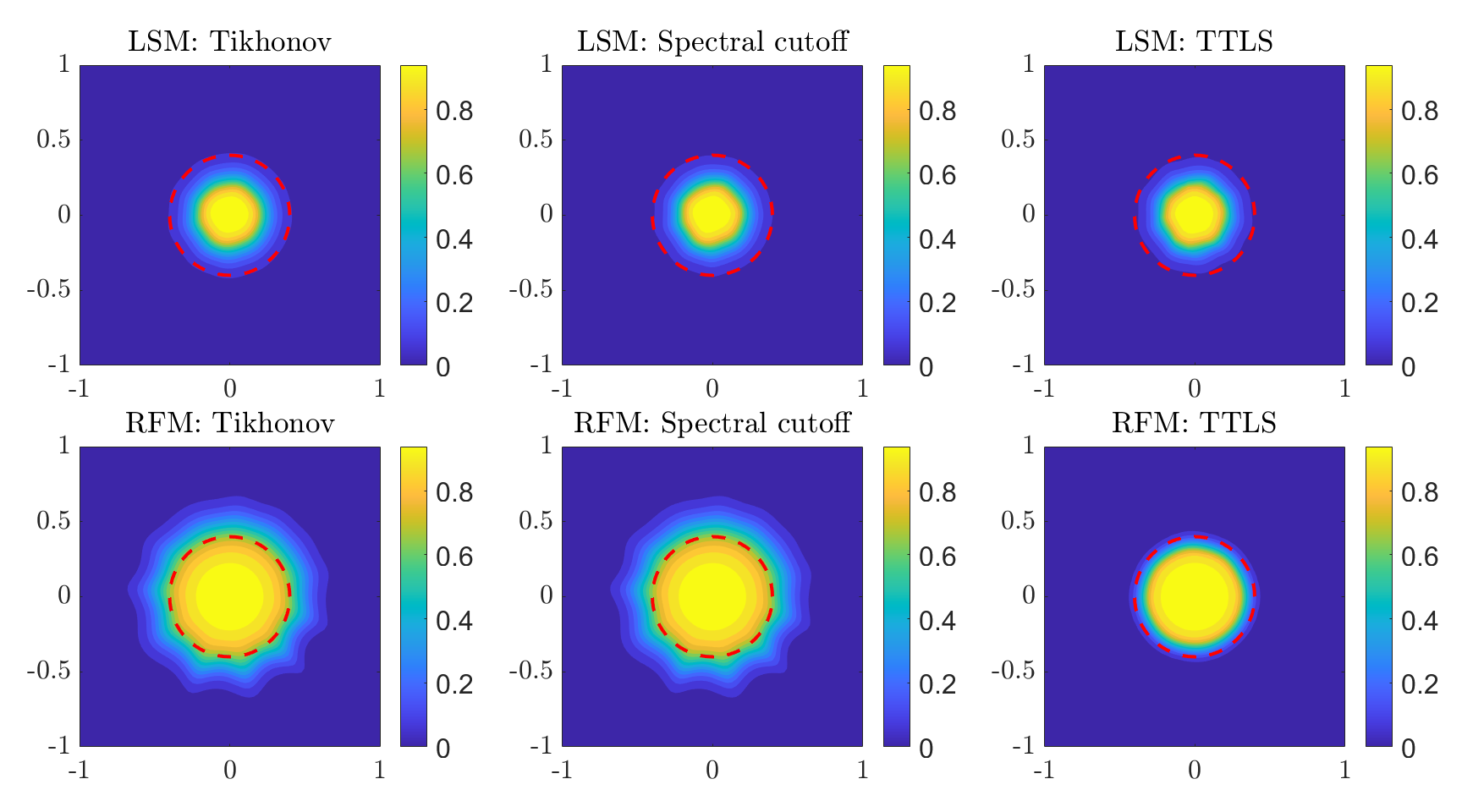}
\caption{Reconstruction of a circular region (radius $\rho = 0.4$) shown by the dashed red line, where the exterior conductivity is $\sigma_{\text{out}} = 1$ and the interior conductivity is $\sigma_{\text{in}} = 10$. The reconstruction is obtained using both the LSM and RFM methods, with $1\%$ additive noise included in the data matrix ($\delta = 0.01$). The boundary data are sampled at $N = 64$ equiangular points.}
\label{fig: recon-circle-0.4-two sigmas_N64}
\end{figure}

\begin{figure}[!ht]
\centering 
\includegraphics[width=1\linewidth]{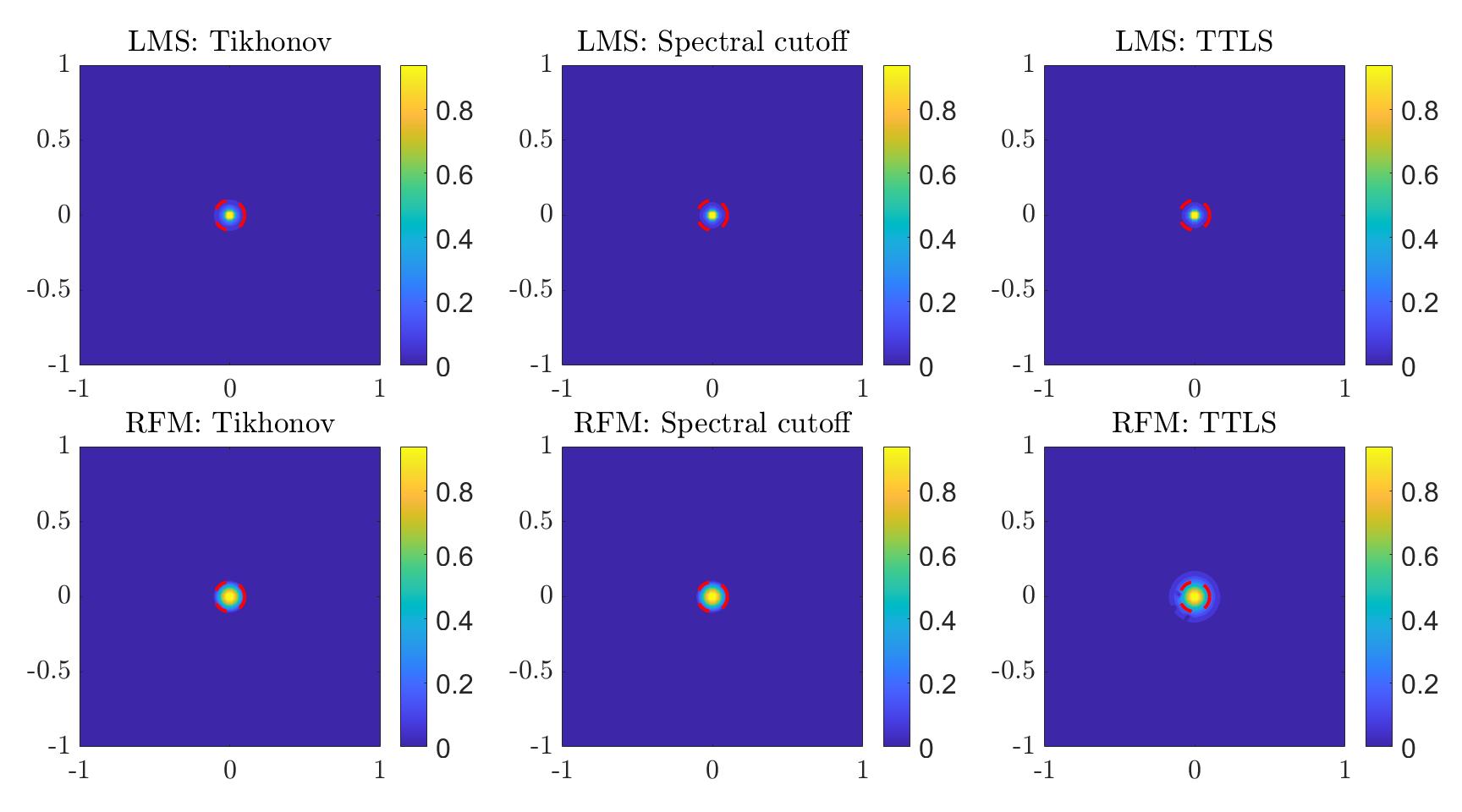}
\caption{Reconstruction of a circular region (radius $\rho = 0.1$) shown by the dashed red line, with conductivity $\sigma = 1$, using both the LSM and RFM methods. The data matrix includes $1\%$ additive noise ($\delta = 0.01$).}\label{fig: recon-circle-0.1}
\end{figure}

\section{Conclusions}\label{end}
In this paper, we investigated two qualitative methods for solving the inverse shape problem in EIT with Robin boundary conditions imposed on both the outer measurement surface and the interior interface. Specifically, we have analyzed the Linear Sampling Method and the Regularized Factorization Method, deriving the corresponding imaging functionals that enable the reconstruction of the region of interest $D$ from the difference of the Robin-to-Dirichlet operators $(M-M_0)$. These approaches provide fast and accurate reconstructions with minimal a priori information about $D$, and exhibit stability with respect to noisy data.

Additionally, we have introduced the Truncated Total Least Squares (TTLS) regularization approach into the qualitative reconstruction framework. To the best of our knowledge, this represents the first application of TTLS to qualitative EIT and inverse shape problems with Robin conditions. The TTLS scheme offers a robust alternative to classical Tikhonov and spectral cut-off regularizations, producing smoother and more stable imaging functionals, particularly in the presence of measurement noise.

Future research directions include extending the framework to the inverse transmission parameter problem to develop a non-iterative method for recovering $\gamma$, as well as exploring the Direct Sampling Method (see, e.g., \cite{DSM-EIT, DSM-DOT, DSM-nf}) within the same Robin-type setting. Another natural extension is to provide analysis when considering non-constant and distinct conductivities (inside and outside the inclusion), thereby accommodating more physically realistic configurations in EIT and related non-destructive evaluation applications.

\section*{Acknowledgments}
M.I.E. was supported through a Karen EDGE Fellowship. G.G. was supported by NSF RTG DMS-2135998. M.I.E. and G.G. gratefully acknowledge the Society for Industrial and Applied Mathematics (SIAM) and the SIAM Postdoctoral Support Program, established by Martin Golubitsky and Barbara Keyfitz, for helping to make this work possible.

\bibliographystyle{plain}
\bibliography{references.bib}

\appendix

\section{Appendix}\label{app: A}
We present the method we used for discretizing $(M-M_0)$ in the unit disk. We assume $\partial D$ is given by $\rho (\text{cos} ( \varphi ), \text{sin} (\varphi))$ for some constant $\rho \in (0,1)$. Since $\Omega$ is assumed to be the unit disk in $\mathbb{R}^2$, we make the ansatz that the electrostatic potential $u(r,\varphi)$ has the following series representation
\begin{equation}\label{taurus}
u(r , \varphi) = a_0 + b_0 \, \text{ln} \, r + \sum_{|n|=1}^{\infty} \left[ a_n r^{|n|} + b_n r^{-|n|} \right] \text{e}^{\text{i}n \varphi} \quad \text{in} \quad \Omega \backslash D,
\end{equation}
which is harmonic in the annular region and also
$$u(r , \varphi) = c_0 + \sum_{|n|=1}^{\infty} c_n r^{|n|} \text{e}^{\text{i}n \varphi} \quad \text{in} \quad D,$$
which is harmonic in the circular region. For simplicity, we assume that the transmission parameter $\gamma >0$ is constant. Thus, we are able to determine the Fourier coefficients $a_n$ and $b_n$ by using the boundary conditions at $r=1$ and $r = \rho$ given by
$$ \sigma \partial_{r} u(1, \varphi )+u(1,\varphi) = f(\varphi), \quad u^{+} (\rho , \varphi ) = u^{-} ( \rho , \varphi ), \quad \text{and}$$
$$ \sigma\partial_{r} u^{+} (\rho , \varphi ) - \sigma\partial_{r} u^{-} (\rho , \varphi ) = \gamma u(\rho , \varphi).$$
We let $f_n$ for $n \in \mathbb{Z}$ denote the Fourier coefficients of $f$. Note that the boundary condition at $r=1$ above gives that $$ a_0 +\sigma b_0 = f_0 \quad \text{for}\quad n=0 \quad \text{and} \quad \sigma|n|a_n+a_n-\sigma|n|b_n+b_n = f_n \quad \text{for all} \enspace n \neq 0.$$
The first boundary conditions at $r = \rho$ give 
$$ a_0+b_0ln(\rho)=c_0 \quad \text{for}\quad n=0  \quad \text{and} \quad a_n\rho^{-|n|}+b_n^{-|n|}-c_n\rho^{|n|} = 0\quad \text{for all} \enspace n \neq 0. $$ Using the Robin transmission condition at the boundary $\partial D$, and after some calculations, we get that for $n\neq 0$
$$ a_0 \frac{b_0}{\rho}-\gamma c_0=0 \quad \text{and} \quad \sigma|n|a_n\rho^{|n|-1}-\sigma b_n^{-|n|-1}-c_n(\sigma|n|\rho^{|n|-1}+\gamma\rho^{|n|})=0 \quad \text{for all} \enspace n \neq 0.$$
For $n=0$, we solve the following system 
$$\begin{bmatrix}
1 & \sigma & 0\\
1 & ln\rho & -1\\
0& \frac{\sigma}{\rho} & -\gamma

\end{bmatrix}\begin{bmatrix}
    a_0\\
    b_0\\
    c_0\\
\end{bmatrix}=\begin{bmatrix}
    f_0\\
    0\\
    0
\end{bmatrix},$$
and $n\neq0$ we solve

$$\begin{bmatrix}
\sigma|n|+1 & -\sigma|n|+1 & 0\\
\rho^{|n|} & \rho^{-|n|} & -\rho^{|n|}\\
\sigma|n|\rho^{|n|-1}& -\sigma\rho^{-|n|-1} & -(\sigma|n|\rho^{|n|-1}+\gamma\rho^{|n|})

\end{bmatrix}\begin{bmatrix}
    a_n\\
    b_n\\
    c_n\\
\end{bmatrix}=\begin{bmatrix}
    f_n\\
    0\\
    0
\end{bmatrix}.$$ Thus, one concludes that $a_0=\alpha_0f_0, b_0=\beta_0f_0,$ and $c_0=\omega_0f_0$ where each of the $\alpha_0$, $\beta_0$, and $\omega_0$ depend on $\sigma$, $\rho$, and/or $\gamma.$ Similarly, when solving for $a_n$, $b_n$, and $c_n$, one finds that they have also the form $a_n=\alpha_nf_n,$ $b_n=\beta_nf_n$, and $c_n=\omega_nf_n$, where $\alpha_n,\beta_n$, and $\omega_n$ depend on $\sigma$, $\rho$ and/or $\gamma.$ Plugging the sequences into \eqref{taurus} gives that the corresponding current on the boundary of the unit disk is given by
\begin{equation}\label{unormal}
u(1 , \varphi ) = \alpha_0 f_0 + \sum_{|n|=1}^{\infty} (\alpha_n + \beta_n) f_n \text{e}^{\text{i}n \varphi}.
\end{equation}

The electrostatic potential for the material without a corroded interface is given by 
\[
u_0 (r , \varphi ) = \kappa_0 + \sum_{|n|=1}^{\infty} \kappa_n r^{|n|} \text{e}^{\text{i}n \varphi},\]
and by applying the boundary condition at $r=1$
$$\sigma \partial_r u_0(1,\varphi ) + u_0 (1,\varphi) = f(\varphi),$$  
we obtain
 \[ \sigma \sum_{|n|=1}^{\infty}|n| \kappa_n e^{in \varphi} + \kappa_0 + \sum_{|n|=1}^{\infty} \kappa_n \text{e}^{\text{i}n \varphi} = f_0 + \sum_{|n|=1}^{\infty}f_n \text{e}^{\text{i}n \varphi}. \]
Therefore, $\kappa_0 = f_0$ and $\kappa_n = \frac{f_n}{\sigma|n|+1}$ for $n \neq 0$. Thus,
\begin{equation}\label{u0normal}
u_0 (1,\varphi ) = f_0 + \sum \frac{f_n}{\sigma |n|+1} e^{in \varphi}.
\end{equation}

Subtracting equation \eqref{u0normal} from \eqref{unormal} gives a series representation of the current gap operator. By interchanging summation with integration, we obtain 
$$(u - u_0)(\varphi) =  (M-M_0)f(\varphi) = \frac{1}{2 \pi} \int_{0}^{2 \pi} K (\varphi , \phi) f ( \phi) \, \text{d} \phi, \quad $$ where $$K(\varphi , \phi ) = (\sigma_0-1) + \sum_{|n|=1}^{\infty} \left(\alpha_n + \beta_n - \frac{1}{\sigma |n| +1}\right) \text{e}^{\text{i}n (\varphi - \phi)}.$$

\end{document}